\documentclass[a4paper,smallextended,american,final,natbib]{article}
\usepackage{cancel}
\usepackage[final,todonotes={textsize=footnotesize}]{changes}
\definechangesauthor[name={Sascha Kurz}, color=red]{SK}
\definechangesauthor[name={Jörg Rambau}, color=blue]{JR}
\setdeletedmarkup{\ifmmode\xcancel{#1}\else\sout{#1}\fi}
\usepackage{amsmath,amsthm,amssymb}
\usepackage{microtype}
\usepackage{enumerate}
\usepackage{graphicx}
\usepackage{xcolor}
\usepackage{float}
\usepackage[boxed,vlined,longend,linesnumbered]{algorithm2e}
\SetAlFnt{\footnotesize}
\SetFuncSty{sc}
\usepackage{geometry}

\usepackage{array,tabularx,booktabs}
\newcolumntype{R}{>{\raggedleft\arraybackslash}X}
\newcolumntype{C}{>{\centering\arraybackslash}X}

\usepackage{url}

\newcommand{\abs}[1]{\ensuremath{\lvert #1 \rvert}}
\newcommand{\lottypesatbranchfunction}{L}
\newcommand{\lottypesatbranch}{L(b)}

\newcommand{\lottypessubsetatbranch}[1]{L_{#1}(b)}
\newcommand{\multsatbranchlottypefunction}{M}
\newcommand{\multsat}[3][]{M_{#1}(#2, #3)}
\newcommand{\multsatbranchlottype}{\multsat{b}{l}}

\newcommand{\multssubatbranchlottype}[1]{\multsat[#1]{b}{l}}

\newcommand{\slackvariable}{p}
\newcommand{\slackcost}{P}
\newcommand{\STATICCPLEX}{STATIC\textunderscore{}CPLEX}

\theoremstyle{theorem}
\newtheorem{observation}{Observation}
\newtheorem{theorem}{Theorem}

\newtheorem{remark}{Remark}
\theoremstyle{definition}
\newtheorem{definition}{Definition}

\let\emptyset\varnothing

\begin{document}
%% \color{gray}
\allowdisplaybreaks

%%\date{\today}
\date{}

\title{An exact column-generation approach for the lot-type design problem}

\author{Miriam Kie\ss ling \and Sascha Kurz \and J\"org Rambau
  \\[3mm]Lehrstuhl f\"ur Wirtschaftsmathematik
  \\Universit\"at Bayreuth\\Germany.
  \\ Tel.: +49-921-557353, Fax: +49-921-557352
  \\\texttt{miriam.kiessling@googlemail.com}
  \\\texttt{$\{$sascha.kurz,joerg.rambau$\}$@uni-bayreuth.de}}

%%\listofchanges[title={List of Changes}, style=compactsummary]

\maketitle

\begin{abstract}
  We consider a fashion discounter distributing its many bran\-ches with integral multiples from a set of available lot-types.
  For the problem of approximating the branch and size dependent demand using those lots we propose a tailored exact column
  generation approach assisted by fast algorithms for intrinsic subproblems, which turns out to be very efficient on our
  real-world instances as well as on random instances.\\

  \noindent
  \textbf{Keywords:} $p$-median; facility location; lot-type design; column generation\\
  \textbf{Mathematics subjects classification:} 90C06; 90C90; 90B99
\end{abstract}

\section{Introduction}
\label{sec_introduction}

\noindent
Due to small profit margins of most fashion discounters, applying OR
methods is mandatory for them. In order to reduce the handling costs
and the error proneness in the central warehouse, our business partner
orders all products in multiples of so-called \textit{lot-types} from
the suppliers and distributes them without any replenishing to its
branches. A lot-type specifies a number of pieces of a product for
each available size, e.g., lot-type $(1,2,2,1)$ means two pieces of
size $\text{M}$ and $\text{L}$, one piece of size $\text{S}$ and
$\text{XL}$, if the sizes are
$(\text{S},\text{M},\text{L},\text{XL})$.

We want to solve the following approximation problem: which (integral)
multiples of which (integral) lot-types should be supplied to a set of
branches in order to meet a (fractional) expected demand as closely as
possible? We call this specific demand approximation problem the
\textit{lot-type design problem (LDP)} in~\cite{p_median}. In that
paper, also a basic model for the LDP was introduced, accompanied by
an integer linear programming formulation and a tailored heuristic,
which turned out to perform very well for the real-world data of our
partner.

\added[id=JR, comment={made clear why the old paper did not solve the
  problem for us}]{The problem that remained unaddressed in~\cite{p_median}
  is the following:} For many practical instances the set of
applicable lot-types and thus the number of variables is so large that
the ILP formulation from~\cite{p_median} cannot be solved
directly.\footnote{E.g.\ for $12$ different sizes, which is reasonable
  for lingerie or children's clothing, there are $1\,159\,533\,584$
  different lot-types, if we assume that there should be at most $5$
  items of each size and that the total number of items in a lot-type
  should be between $12$ and $30$.}  \added[id=JR, comment={put a
  preview of the contribution right here}]{In this paper, we
  therefore propose a new algorithm \emph{Augmented Subproblem
    Grinding (ASG)} based on dynamic model generation by the
  branch-and-price principle.  The main result is twofold,
  experimentally obtained on a benchmark set of 860 real-world
  instances from the running production in 2011: First, ASG is able to
  solve LDPs with millions and billions of lot-types (for details see
  below) in 10 to 100 seconds to optimality.  Second, for smaller
  instances ASG is faster by a factor of 100 (more for larger
  instances) than the solution of the statically generated ILP model
  for LDP with a state-of-the-art MILP solver (in this paper:
  \texttt{cplex 12.9}).}

\subsection{Related work}\label{sec:related-work}

\added[id=JR, comment={rewrote the literature review for apparel
  supply chain management}]{We start with a brief overview of related
  literature from the application side.  Apparel supply chains pose
  many challenges.  A generic survey would be beyond the scope of this
  paper.  An overview of some aspects (though not specifically related
  to the special subject of this paper) can be found in
  \cite{Ahsan+Azeem:ApparelSCM:2010}. Let us mention two very
  difficult partial challenges. The first problem is the estimation of
  the demand, since many apparel products are never
  replenished. Therefore, the sales data only yield a right-censored
  observation of the demand.  Moreover, apparel comes in different
  sizes.  A branch of a fashion discounter is supplied by only very
  few (one, two, or three) pieces per size, and the whole supply must
  have been sold at the end of the sales-period (12 weeks for our
  partner), possibly by heavy mark-downs.  Thus, it is difficult to know, how
  and when the sales data should be measured.  The second problem is
  the integrality of the supply in the presence of small fractional
  expected demands.  That means, that even if the fractional expected
  demand for an item in a size were known, we face the approximation
  problem of which integral supply is best for which expected demand.
  Our industrial partner uses a system of ordering pre-packed
  assortments (lots) of items with only a small variation of possible
  size-combinations (lot-types).  This restricts the approximation
  problem even further (the LDP). See, e.g., \cite{Kurz+etal:TDI:2015}
  for a discussion and a field study of both aspects combined.}

\added[id=JR, comment={rewrote the literature review}]{The exact
  lot-type design problem was introduced for the first time by two of
  the authors in \cite{p_median}.  Later, it was connected to the
  sales-process more accurately by modeling the mark-downs explicitly
  as recourse actions of a two-stage stochastic program in
  \cite{ubt_eref761}.  Since in that paper, the solution of the LDP is
  an important subroutine, the investigations in this paper have
  direct consequences in the applicability of the results
  in~\cite{ubt_eref761} on instances with many applicable lot-types.}

\added[id=JR, comment={rewrote the literature review}]{The LDP is a
  special case of a more general problem setting: assortment
  optimization asks for good combinations of items with arbitrary
  attributes (not only size), and the customer substitutes between
  them, e.g., buys a red instead of a blue T-shirt.  See
  \cite{Fisher+Vaidyanathan:AssortmentDemand:2014} for various aspects
  of the problem including demand estimation and heuristic
  algorithms. In \cite{lbs1085} it is shown that the general
  assortment problem is difficult and approximation algorithms are
  analyzed.}

\added[id=JR, comment={rewrote the literature review}]{Let us now turn
  to related literature from the methodological point of view, mainly
  mixed integer linear programming (MILP)}

\deleted[id=JR, comment={the contribution is now in an extra
  section}]{In order to overcome the integrality gap of the ILP
  relaxation we propose a tailored branching scheme complemented by
  the use of additional cover cuts. This results in an exact column
  generation approach for the LDP, which is enhanced by properly
  chosen algorithms for important subproblems.  We apply this
  algorithm to a stochastic version SLDP of the LDP, where the
  expectation over more than one demand-scenario is optimized.  The
  SLDP is of the same form as the LDP, and thus all optimization
  techniques viable for the LDP immediately apply to the SLDP.  Since
  using the SLDP instead of the LDP can make decisions more robust
  against forecasting errors, we based our investigations in this
  paper on the SLDP.}

\added[id=JR, comment={comments on BP}]{Our new method is a
  specialization of the \emph{branch-and-price} principle for
  mixed-integer linear programming problems (MILP).  In the
  branch-and-price framework, not all variables are added to an MILP
  in the beginning.  Instead, in each branch-and-bound node of the
  MILP-solution process, variables are added dynamically by solving
  the pricing problem until a computational optimality proof can be
  completed.  A standard survey for branch-and-price algorithms for
  large-scale integer programming problems can be found
  in~\cite{ml:LuebbeckeDesrosiers:05}.  One of the most influential
  early systematic discussions for the branch-and-price principle for
  MILP was presented in \cite{0979.90092}.  In the same spirit, in
  \cite{branch_and_price_generic_scheme} a (quite sophisticated)
  branching scheme based on a classification of columns is formulated
  in a problem-independent way.  Our new method by-passes the
  challenges attacked in that paper and is therefore, in our opinion,
  easier to implement.
% Extending branch-and-price by cutting plane methods we get to
% branch-price-and-cut algorithms. General results can be found in
% \citep{bpc-algo}.
}

\added[id=JR, comment={comments on BPC}]{If not only variables but
  also restrictions are added dynamically, one is concerned with
  \emph{branch-price-and-cut} algorithms (see
  \cite{rothenbacher2016branch} for a more recent practical
  application, \cite{bpc-algo} for a methodological survey and
  \cite{Sadykov:PhDThesis:2019} for a more recent comprehensive
  treatment).  There are two principally different reasons for adding
  restrictions dynamically.  First, they can be inevitable, when there
  are model restrictions whose support solely consists of variables
  that are dynamically generated (model restrictions).  Such
  restrictions leading to \emph{row-dependent columns} are discussed,
  e.g., in \cite{note,sim}.  Second, they can be desirable, when they
  help to close the integrality gap in the master problem (cuts).  In
  the second case, one speaks of \emph{robust} cuts if the cuts do not
  change the complexity of the pricing problem (see
  \cite{Pessoa2008}).  In contrast to this, \emph{non-robust} cuts
  need to be observed in the pricing problem and can increase its
  complexity (see \cite{675a55cc4a9c45cc9bdcbaad4dc58755}).  In the
  case of the LDP, we are concerned with inevitable and new restrictions
  that have a significant influence on the pricing problem, as we will
  see below.  In \cite[Section~3.2]{bpc-algo} it is discussed how, in
  principle, the problem of determining reasonable dual variables for
  future cuts in the master problem should be attacked.  However, this
  is essentially a problem that until today is more successfully
  solved for each particular application separately.  One of our main
  results is a solution to this \emph{dual lifting problem} for the
  LDP.}

\added[id=JR, comment={comments on heuristics}]{Another important
  partial problem is the determination of feasible integer solutions.
  If there is no tight problem-specific heuristic available, generic
  primal heuristics can be used, e.g., based on diving down the
  branch-and-bound tree like in \cite{sadykov2019primal}.  In our
  case, a fast and tight heuristics is fortunately known from
  \cite{p_median}.  In \cite{vaclavik2018accelerating} it was
  demonstrated how machine learning can accelerate the time spent in
  the pricing problem by learning a dual bound from previous
  iterations.  Since for the LDP the time spent in the pricing problem
  is not dominant, we have not (yet) tried this idea for the LDP.}

\added[id=JR, comment={comments on $p$-median}]{The LDP is related to
  the $p$-median and the facility location problem: for computational
  results on large instances of the $p$-median problem we refer the
  reader to
  \cite{pre05131061,1170.90521,1163.90607,rebreyend2015computational}.
  Other algorithmic approaches can, e.g., be found in
  \cite{hansen1997variable,mladenovic2007p,hale2017lagrangian,drezner2015new,agra2017decomposition,fast2017branch,marzouk2016branch}.
  The additional cardinality constraint of the LDP prevents a direct
  application of specialized $p$-median algorithms.  Therefore, we
  follow the concept of the more generic branch-and-price framework.}

\deleted[id=JR, comment={replaced this by a more detailed review}]{%
  Branch-and-price algorithms are common \ldots{} have been addressed
  in [\ldots].  }

\subsection{Contribution}\label{sec:contribution}

\added[id=JR, comment={put the contributions here in an extra
  section}]{We present a new algorithmic procedure named
  \emph{Augmented Subproblem Grinding (ASG)} based on the
  branch-and-price principle, albeit with a new strategy.  The new
  results are:}
\begin{itemize}
\item \added[id=JR, comment={the theoretical
    contribution}]{Theoretically, we construct for the LDP a
    characteristic lifting of the dual variables that yields, in a
    certain sense, the tightest possible information about promising
    new columns.}
\item \added[id=JR, comment={the algorithmic
    contribution}]{Algorithmically, based on the previous contribution, we devise a
    column-generation procedure that generates promising combinations of
    columns instead of sets of unrelated columns.  These combinations
    avoid the generation of columns whose negative reduced costs can be
    compensated by an easy modification of the duals of the reduced
    master problem.}
\item \added[id=JR, comment={the practical
    contribution}]{Practically, ASG uses a (to the best of our
    knowledge) new branching scheme: at any time there are only two
    branches, one of which is solved immediately by the black-box
    MILP-solver, whereas the other excludes all solutions of the
    solved branch by a single set-covering constraint with dynamically
    growing support.  Thus, there no branch-and-price-and-bound tree
    must be maintained.  Consequently, ASG is easier to implement than
    ordinary branch-and-price.  Moreover, it exploits the employed
    black-box MILP-solver and its branching intelligence.  This way,
    ASG will take profit of any improvement in the MILP-solver
    technology without modifications in its own code.}
\item \added[id=JR, comment={the experimental
    contribution}]{Experimentally, the efficiency and effectiveness of
    ASG is shown based on 860 real-world instances from the daily
    production of our industry partner from 2011: For the task to find
    an optimal solution and prove its optimality, ASG reduces the
    cpu-times to 0.67\,\% and the numbers of columns to 0.18\,\% on
    average compared to solving the full ILP model by a
    state-of-the-art MILP-solver.  Moreover, for the largest 100
    instances with 134,596 through 1,198,774,720 applicable lot-types
    (unsolved prior to this paper), ASG needs less than 30 seconds and
    10,000 columns on average.  This makes ASG perfectly applicable in
    practice.  We confirm the success by a second randomized benchmark
    set in order to avoid possible artifacts caused by a hidden
    structure of our real-world instances.}
\end{itemize}

\subsection{Outline of the paper}\label{sec:outline-paper}

A formal problem statement is given in
Section~\ref{sec_formal_problem_statement}, followed by an ILP model
in Section~\ref{sec:modelling}. In Section~\ref{sec:theory} we discuss the 
theoretical foundation of our algorithm that is presented in
Section~\ref{sec_column_generation}. We show computational results on
real-world data and on random data in
Section~\ref{sec_results_column_generation}, before we conclude with
Section~\ref{sec_conclusion}.

\section{Formal problem statement}
\label{sec_formal_problem_statement}

\noindent
We consider the distribution of supply for a single product and start
with the formal problem statement in the deterministic
context\added[id=JR, comment={separated LDP from SLDP}]{ before we
  extend it to a stochastic version}.

\subsection{The deterministic lot-type design problem (LDP)}
\label{sec:determ-lot-type}

\textit{Data.} Let $\mathcal{B}$ be the set of branches, $\mathcal{S}$
be the set of sizes, and $\mathcal{M}\subset\mathbb{N}$ be an interval
of possible multiples. A \emph{lot-type} is a vector
$l = (l_s)_{s\in\mathcal{S}}\in\mathbb{N}^{|\mathcal{S}|}$.
\added[id=JR, comment={separated formal problem statement from a
  remark explaining the motivation}]{For given integers
  $\min_c \le \max_c$ (the \emph{component bounds}) and
  $\min_t \le \max_t$ (the \emph{type-bounds}), a lot-type $l$ is
  \emph{applicable}, if $\min_c \le l_s \le \max_c$ for all
  $s \in \mathcal{S}$ and
  $\min_t\le \abs{l} := \sum_{s\in\mathcal{S}} l_s\le \max_t$.  The
  set of applicable lot-types is called a \emph{set of lot-types
    parametrized by $\min_c$, $\max_c$, $\min_t$, and $\max_t$} and is
  denoted by $\mathcal{L}$.  Note, that specifying a parametrized set
  of lot-types and specifying an explicit list of lot-types differ in
  terms of the complexity of models.  A model that is polynomial in
  the number of lot-types is polynomially-sized if the input is an
  explicit list of lot-types whereas it may be exponentially-sized if
  the input is a parametrized set of lot-types.}

\deleted[id=JR, comment={integrated into the formal problem
  statement}]{\color{red}For practical as well for the subsequent
  algorithmical purposes the set of lot-types that we are
  considering here needs to be further restricted. Let $\min_t$ and
  $\max_t$ be integers that lower or upper bound the total number of
  items in a lot-type, respectively , i.e.,
  $\min_t\le \sum_{s\in\mathcal{S}} l_s\le \max_t$.}

\added[id=SK]{The motivation \added[id=JR, comment={connected to the
    attribute ``parametrized''}]{for parametrized sets of lot-types}
  is as follows.}  \replaced[id=JR, comment={simplified
  sentences}]{The cost reduction induced by ordering apparel in
  pre-packed lots in the Far East (low wage) results from the
  reduction of the number of picks in the central warehouse in Europe
  (high wage).  On the one hand, this reduction is only significant if
  there are enough pieces in the used lot-type.  This can be enforced
  by setting $\min_t$ large enough.  On the other hand, a single pick
  of a lot is only possible if there are not too many pieces in the
  used lot-type.  This can be enforced by setting $\max_t$ small
  enough.  Moreover, if a product is advertised in a newspaper or the
  like, each branch must have at least one piece in each size by legal
  regulations.  This can be guaranteed by setting $\min_c$ to one.
  For the sake of a clean presentation at the sales start, it may be
  desirable to enforce that no single size is extremely
  over-represented.  This can be enforced by setting $\max_c$ small
  enough.  If $\min_c = 0$ and $\max_c = \max_t$, then $\min_c$ and
  $\max_c$ do not impose any restriction beyond the restrictions
  imposed by $\min_t$ and~$\max_t$.  In practice, the lot-type
  parameters $\min_t$, $\max_t$, $\min_c$, and $\max_c$ depend on the
  product.  For example, it is easily possible to pick a lot
  containing ten T-shirts whereas it is undesirable to pick a lot of
  ten winter coats. }{Since the main advantage of using lot-types lies
  in the reduction of the number of picks in the central warehouse, we
  should guarantee, that this effect does not dwindle away by
  selecting lot-types with too few items, which can be controlled by a
  suitable value for $\min_t$. There are practical reasons for the
  parameter $\max_t$, too: combining too many winter coats in a lot
  would cause serious handling problems. Similarly to restrictions on
  the total number of items in a lot-type, we additionally introduce
  restrictions for the number of items of the same size. So, let
  $\min_c$ and $\max_c$ be integers that lower or upper bound the
  number of items of a lot-type in each size, respectively, i.e.,
  $\min_c\le l_s\le\max_c$ for all $s\in\mathcal{S}$. Again an example
  why those kinds of restrictions may be practically relevant. By
  setting $\min_c=1$ we can enforce that each branch is supplied in
  each size with at least one item, a requirement which legally arises
  for advertised products. If there is no practical reason for the
  choice of $\max_c$, we can make the constraint ineffective by
  setting $\max_c=\max_t$. To sum up, we call a lot-type $l$
  \emph{applicable} if $\min_c\le l_s\le\max_c$ for all
  $s\in\mathcal{S}$ and
  $\min_t\le \sum_{s\in\mathcal{S}} l_s\le \max_t$ and denote the set
  of applicable lot-types by $\mathcal{L}$. For the managers the usage
  of a parameterized set of lot-types, as characterized above, avoids
  the necessity of explicitly listing the set of suitable lot-types,
  whose number can get quite large, as we will see later on. Of
  course, it is possible to remove some specific lot-types from the
  set of applicable lot-types in practice. Since the corresponding
  algorithmic modifications are minor and in order to ease the
  notation, we do not go into details here.}

There is an upper bound $\overline{I}$ and a lower bound
$\underline{I}$ given on the total supply over all branches and sizes.
Moreover, there is an upper bound $k \in \mathbb{N}$ on the number of
lot-types used. By $d_{b,s}\in\mathbb{Q}_{\ge 0}$ we denote the
\replaced[id=JR, comment={no expectation in the LDP}]{deterministic}{
  expected} demand at branch~$b$ in size~$s$.

\deleted[id=JR,comment={chose to rewrite the problem statement to
  contain both SLDP and LDP as well as their
  relationship}]{  
  \color{red}Here, {\lq\lq}expected
  demand{\rq\rq} either means just a deterministic demand or the
  expectation
  $\sum_{\Omega\in \Omega} p^\omega\cdot d_{b,s}^\omega$ in the
  augmented stochastic model described at the end of this
  section.}

\textit{Decisions.} Consider an assignment of a unique lot-type
$l(b) \in \mathcal{L}$ and an assignment of a unique multiplicity
$m(b)\in\mathcal{M}$ to each branch~$b \in \mathcal{B}$.  These data
specify that $m(b)$ lots of lot-type $l(b)$ are to be delivered to
branch~$b$.  These decisions induce inventories
$I_{b,s}(l, m) := m(b) l(b)_s$ for each branch $b \in \mathcal{B}$ and
each size $s \in \mathcal{S}$ and a total inventory of
$I(l, m) := \sum_{b \in \mathcal{B}} m(b) \abs{l(b)}$ over all
branches and sizes.

\emph{Objective.} The goal is to find a subset
$L \subseteq \mathcal{L}$ of at most $k$ lot-types and assignments
\replaced[id=JR, comment={the choice for each branch must be in the
  selected set}]{$l(b) \in L$}{$l(b) \in \mathcal{L}$} and
$m(b) \in \mathcal{M}$ such that the total supply is within the bounds
$\bigl[\underline{I}, \overline{I}\bigr]$, and the total deviation
between \replaced[id=JR, comment={unified}]{supply}{inventory} and
demand is minimized.  \added[id=JR, comment={formalized the
  definitions}]{More specifically, the \emph{deviation}
  $\Delta_{b,s}(l, m)$ of supply and demand for branch~$b$ and
  size~$s$ given a decision $l(b)$, $m(b)$ for all $b \in \mathcal{B}$
  is defined by
  $\Delta_{b,s}(l, m) := \lvert d_{b,s} - I_{b,s}(l, m) \rvert$.  The
  \emph{cost} of the decisions $l(b)$, $m(b)$ for branch
  $b \in \mathcal{B}$ is defined as
  $c_b(l, m) := \sum_{s \in \mathcal{S}} \Delta_{b,s}(l, m)$.  The
  goal is to minimize the total cost
  $c(l, m) := \sum_{b \in \mathcal{B}} c_b(l, m)$.}

\replaced[id=JR, comment={minor rephrasing}]{The resulting
  optimization problem is}{We call this optimization problem} the
\emph{Lot-Type Design Problem (LDP)}.  \added[id=JR,
comment={emphasized the difference to the LDP paper}]{It is strongly
  related to the LDP in \cite{p_median}.  The difference is that in
  \cite{p_median} the set of applicable lot-types was given by an
  explicit list~$\mathcal{L}$.}  \deleted[id=JR, comment={was shifted
  to an earlier paragraph}]{Using the introduced decision variables
  $L$, $l(b)$, and $m(b)$, we can express the relevant
  decision-dependent entities as follows.  The \emph{inventory of
    branch $b$ in size~$s$} given assignments $l(b)$ and $m(b)$ is
  given by $I_{b,s}(l,m) = m(b) l(b)_s$. Moreover, the \emph{total
    supply} resulting from $l(b)$ and $m(b)$ is given by
  $I(l,m) = \sum_{b \in \mathcal{B}}\sum_{s \in \mathcal{S}}
  I_{b,s}(l,m)$.}

\subsection{The stochastic lot-type design problem (SLDP)}
\label{sec:stoch-lot-type}

\replaced[id=JR, comment={separation of LDP and SLDP}]{The SLDP
  differs from the LDP only in the following.  We consider }{This
  deterministic model can slightly be enhanced to a stochastic model
  by considering} a set $\Omega$ of \emph{scenarios} (for the success
of the product). For each scenario $\omega \in \Omega$ we denote by
$p^{\omega}$ its probability and with
$d_{b,s}^{\omega} \in \mathbb{Q}_{\ge 0}$ the demand at branch~$b$ in
size~$s$ in scenario~$\omega$ for all $b\in\mathcal{B}$ and
$s\in\mathcal{S}$. The goal then is to minimize the expected total
deviation between \replaced[id=JR,
comment={unified}]{supply}{inventory} and demand.  \added[id=JR,
comment={state the obvious}]{The remaining ingredients are the same as
  for the LDP. In particular, the LDP and the SLDP differ only in the
  objective function.}

We call this single-stage stochastic optimization problem the
\emph{Stochastic Lot-Type Design Problem (SLDP)}.

\replaced[id=JR,comment={in the experiments, the nominal scenario is
  used; no idea whether it corresponds to the expected demand}] {In
  this paper, we assume for an SLDP that there is a designated
  \emph{nominal scenario $\hat{\omega} \in \Omega$} that models a
  customer demand that is considered ``normal''.  For example, in the
  data of our industrial partner there are usually three scenarios
  that stem from the classification of a product as ``Renner'' (German
  designation for a product in high demand), ``Normal'' (= normal demand), and
  ``Penner'' (German designation for a product in low demand).  In general, if a
  nominal scenario is not given explicitly in an SLDP, the scenario of
  expected demands is added to the set of scenarios and defined as the
  nominal scenario.

  Each SLDP with a nominal scenario induces two LDPs with identical
  constraint sets as the SLDP: The LDP with deterministic demands
  corresponding the the nominal scenario is called the \emph{nominal
    LDP} of the SLDP.  The LDP where we set the deterministic
  deviations $\Delta(l, m)$ to the expected deviations
  $\overline{\Delta}(l, m) := \sum_{\omega \in \Omega} p^{\omega}
  \Delta^{\omega}(l, m)$ with
  $\Delta^{\omega}(l, m) := \sum_{b \in \mathcal{B}} \sum_{s \in
    \mathcal{S}} \Delta^{\omega}_{b, s}(l, m)$ and
  $\Delta^{\omega}_{b, s}(l, m) := \lvert d_{b, s}^{\omega} -
  \added[id=SK]{I_{b, s}(l, m)}\rvert$ is called the \emph{equivalent
    LDP} of the SLDP.  The equivalent LDP of an SLDP is equivalent to
  the SLDP since both constraints and objective functions are
  identical by the linearity of expectation.  In other word, for any
  SLDP the equivalent LDP is a \emph{deterministic equivalent
    program}~\cite[Section~2.4]{1223.90001} for the SLDP.

  In any case, any algorithm for the LDP can be transformed into an
  algorithm for the SLDP by replacing each computation of
  $\Delta_{b, s}(l, m)$ in the LDP algorithm by the computation of
  $\overline{\Delta}_{b, s}(l, m) = \sum_{\omega \in \Omega}
  p^{\omega} \Delta^{\omega}_{b, s}(l, m)$.  This incurs an overhead
  proportional to the number of scenarios.

  In contrast to the equivalent LDP, the nominal LDP, in general, has
  a smaller objective than the SLDP, since the deviation is non-linear
  and convex in the demands.  We will find out that this difference in
  the objective values is significant in
  section~\ref{sec_results_column_generation}.  We will also evaluate
  the standard measure for the relevance of the difference in the
  \emph{optimal solutions} of the SLDP and the nominal LDP: The
  \emph{value of the stochastic solution
    (VSS)}~\cite[Section~4.2]{1223.90001}.  It turns out that this
  difference is far less pronounced.}{ In this paper, we
    draw this assumption, i.e., we set
    $d_{b,s}=\sum_{\Omega\in \Omega} p^\omega\cdot
    d_{b,s}^\omega$. However, the only place where we will see the
    scenarios $\omega$ again is in the definition of the constants
    $c_{b,l,m}$ in the ILP model of the next section. From the
    algorithmic point of view, given our assumption, there is no
    difference between SLDP and LDP. From a practical point of view,
    the use of scenarios is highly relevant, since the turnover of
    fashion products can significantly depend on the weather, which
    can e.g.\ be concluded from the amount of money fashion
    discounters spend for very accurate weather forecasts. Considering
    markdown decisions are the most natural candidates for second
    stage decisions. However, we do not complicate our model in that
    way and only use the enhanced deviations function
    $\Delta(l,m)=\sum_{b \in \mathcal{B}} \sum_{s \in \mathcal{S}}
    \sum_{\omega \in \Omega} p^{\omega}\cdot \lvert d_{b,s}^{\omega}
    -\added[id=SK]{I_{b,s}(l,m)}\rvert$ as a reasonable and easy to
    compute proxy for the real costs of {\lq\lq}over-{\rq\rq} or
    {\lq\lq}undersupply{\rq\rq}.  In our experiments, we restrict
    ourselves to the case of the deterministic LDP.}

\section{Modelling}
\label{sec:modelling}

\replaced[id=JR, comment={more precise rephrasing}]{We must choose a
  set~$L$ of selected lot-types with $\abs{L} \le k$}{For each branch
  $b \in B$ we must \added[id=SK]{choose} a set~$L$ of selected
  lot-types with $\abs{L} \le k$} and assign \added[id=JR,
comment={more precise rephrasing}]{to each branch~$b \in \mathcal{B}$}
a lot-type $l=l(b) \in L$ and a multiplicity $m=m(b) \in \mathcal{M}$.

\begin{remark}
  \added[id=JR, comment={made this a displayed remark}]{One possible
    attempt to model the (S)LDP is to use variables for the supplies
    for each size separately.  Such a model is given in
    Appendix~\ref{sec:comp-model-param}.  It has the advantage that
    its size is polynomial in the input size, even for a parametrized
    set of lot-types.  However, it exhibits an extremely large
    integrality gap (see Appendix~\ref{sec:comp-model-param} for
    details).  In this paper, we follow the path of using a tighter
    model at the expense of exponentially many variables.  The model
    was introduced in \cite{p_median} and is polynomially-sized if the
    input is an explicit list of lot-types.  It is exponentially sized
    if the input is a parametrized set of lot-types.  }
\end{remark}

In order to model the SLDP as an integer linear program we use binary
assignment variables $x_{b,l,m}$, \added[id=SK]{i.e, $x_{b,l,m}=1$ if
  branch~b is supplied with lot-type~$l$ in multiplicity~$m$ and
  $x_{b,l,m}=0$ otherwise, see~(\ref{ie_bin_x})}. For the used
lot-types we use binary selection variables $y_l$, \added[id=SK]{i.e.,
  $y_l=1$ if $l \in L$ and $y_l=0$ otherwise,
  see~(\ref{ie_bin_y})}. \replaced[id=JR, comment={was introduced
  before}]{Recall that}{As an abbreviation} we utilize
$\lvert l \rvert:=\sum\limits_{s\in\mathcal{S}}l_s$ for the number of
pieces contained in lot-type~$l$.
\begin{align}
  \label{OrderModel_Target}
  \min && \sum_{b\in\mathcal{B}}\sum_{l\in\mathcal{L}}\sum_{m\in\mathcal{M}} c_{b,l,m}\cdot x_{b,l,m}\\
  \label{OrderModel_EveryBranchOneLottype}
  s.t.  &&
           \sum_{l\in\mathcal{L}}\sum_{m\in\mathcal{M}} x_{b,l,m} &= 1 && \forall b\in\mathcal{B}\\
  \label{OrderModel_Binding}
       && \sum_{m\in\mathcal{M}} x_{b,l,m} & \le y_l && \forall
                                                        b\in\mathcal{B}, l\in\mathcal{L}\\
  \label{OrderModel_UsedLottypes}
       &&
          \sum_{l\in\mathcal{L}} y_l & \le k\\
       &&
          \label{OrderModel_Cardinality}
          \underline{I} \le
          \sum_{b\in\mathcal{B}}\sum_{l\in\mathcal{L}}\sum_{m\in\mathcal{M}}
          m \cdot \lvert l \rvert \cdot x_{b,l,m} &\le \overline{I}\\
       &&
          \label{ie_bin_x}
          x_{b,l,m} & \in\{0,1\} && \forall b\in\mathcal{B}, l\in\mathcal{L}, m\in\mathcal{M}\\
       &&
          \label{ie_bin_y}
          y_l & \in\{0,1\} && \forall l\in\mathcal{L},
\end{align}
where
$c_{b,l,m}= \sum\limits_{\added[id=SK]{\omega\in\Omega}}
p^{\added[id=SK]{\omega}}\cdot \sum\limits_{s \in \mathcal{S}}
\bigl\lvert d_{b, s}^{\added[id=SK]{\omega}} - m \cdot l_s
\bigr\rvert\ge 0$. \added[id=SK]{With this, the stated
  \replaced[id=JR, comment={unified}]{objective function}{target
    function} models the {\lq\lq}costs{\rq\rq} of the deviation
  between supply and demand as motivated in the previous section.}
\added[id=JR, comment={included the LDP}]{The deterministic LDP is
  also covered by defining~$\Omega$ as a single-element set.}
\added[id=SK]{Inequality~(\ref{OrderModel_EveryBranchOneLottype})
  guarantees that each branch is supplied with exactly one lot-type
  and one multiplicity. By inequality~(\ref{OrderModel_Binding}) we
  ensure that $y_l=1$ whenever there is a branch $b$ which is supplied
  with lot-type $l$ in some multiplicity $m$. It may happen that
  $y_l=1$ while no branch is supplied with lot-type
  $l$. \replaced[id=JR, comment={simplified language}]{This is
    logically no problem at this point.  However, later in our
    algorithm, we will add constraints to exclude this case.}{However,
    this does not harm the model.}  By
  inequality~(\ref{OrderModel_UsedLottypes}) we ensure that at most
  $k$ lot-types are used, i.e., those with $y_l=1$. The lower and the
  upper bound on the total supply is modeled by
  inequality~(\ref{OrderModel_Cardinality}).}

\added[id=JR, comment={stressed the fact that we do not need two
  models}]{The model reflects the correspondence between an SLDP and
  its equivalent LDP by the fact that it models both at the same
  time.}

\begin{remark}
  \added[id=JR, comment={motivated the notion of a consistent
    instance}]{It is easy to generate an SLDP with no feasible
    solution.  For example, if supplying each branch with $\min_t$
    pieces exceeds the upper bound $\overline{I}$ on the total supply,
    no feasible lot-type design can be found.  Even a feasible
    instance may be ill-posed from a practical point of view: if the
    (expected) total demand does not satisfy the lower and upper
    bounds on the total supply, then the theoretically most desirable
    supply with exactly the demand will be infeasible for the model.
    Still, there may be feasible solutions missing the demand
    completely.  Just think of the extreme %% extremal 
    case of strictly positive
    demands and $\overline{I} = \min_c = \min_t = 0$; a delivery of
    zero is then formally feasible.  Moreover, the smaller the
    difference between the lower and upper bounds on the total supply
    becomes, the less important is the demand consistency of the
    supply, which is the original goal.  Think again of an extreme
    case of identical lower and upper bounds and only one allowed
    lot-type: then, e.g., if the bound is not divisible by the number
    of branches, there is no feasible solution -- a condition that our
    industrial partner did not intend to impose.  In practice, our
    industrial partner preferred to be notified about such anomalies
    in the input data rather than to receive an artificial solution.
    The usual set-up was to use the expected demand plus/minus 10\%
    for the bounds.}
\end{remark}

\added[id=JR, comment={motivated the experiments with random
  instances}]{In order to avoid anomalies in the input data, we
  specify in the following \emph{consistent instances} of the SLDP and
  the LDP, respectively.  Technically, we define a consistent instance
  in such a way that a natural and easy heuristics (the ALH
  heuristics, see algorithm~\ref{alg:ALH}) always finds a feasible
  solution (see section~\ref{sec:subroutine-SFA} for details).  The
  concept is to report inconsistent instances at the beginning and let
  the core algorithm only deal with SLDP instances that are consistent
  in the following sense:}
\begin{definition}
  An SLDP instance is \emph{consistent}, if
  \begin{enumerate}[(i)]
  \item it is feasible
  \item the total nominal demand lies in the cardinality interval, i.e.,
    $\underline{I} \le \sum_{b \in \mathcal{B}} \sum_{s \in
      \mathcal{S}} d^{\hat{\omega}}_{b, s} \le \overline{I}$ (demand consistency)
  \item the total-cardinality flexibility is at least the
    lot-type-cardinality flexibility, i.e.,
    $\overline{I} - \underline{I} \ge \max_t - \min_t$ (cardinality
    consistency)
  \end{enumerate}
\end{definition}

\section{The theoretical foundations}
\label{sec:theory}

In this subsection, we present the underlying theory in more detail.
\added[id=JR, comment={switched the notation to the LDP
  w.l.o.g.}]{Since for each SLDP we can consider the equivalent LDP,
  we will, without loss of generality, formulate all results in terms
  of the LDP with deterministic demands $d_{b, s}$ to simplify the
  notation.}

By %%considering 
using only a small number of variables, the master problem for the pricing phase (MP) can be 
restricted to a manageable size, resulting in the restricted master problem (RMP).  More specifically: Let
$\mathcal{L}'' \subseteq \mathcal{L}$ be the subset of lot-types used
in the previously solved RSLDP, which is initially empty.  We then consider in
the RMP a (small) subset $\mathcal{L}'\subseteq\mathcal{L}$ of the
lot-types containing~$\mathcal{L}''$. For each branch
$b\in\mathcal{B}$ we consider a subset
$\lottypessubsetatbranch{\mathcal{L}'}=\lottypesatbranch\subseteq\mathcal{L}'$
of these lot-types and for each
$l\in\lottypessubsetatbranch{\mathcal{L}'}$ we consider only a subset
$\multssubatbranchlottype{\mathcal{M}}=\multsatbranchlottype\subseteq\mathcal{M}$
of the multiplicities.

For the following, we denote by SLDP the ILP model of the previous
section augmented by
\begin{itemize}
\item a set covering constraint of the form
  $\sum_{l \in \mathcal{L} \setminus \mathcal{L}'} y_l +
  \slackvariable \ge 1$ with a slack variable~\replaced[id=JR,
  comment={we used latin characters for the primal
    throughout}]{$\slackvariable$}{$\sigma$} \replaced[id=JR,
  comment={added the mnemonic for the variable notation}]{whose use is
    essentially \emph{prohibited} by a very large cost coefficient
    $\slackcost$, set to}{whose cost coefficient~$\slackcost$ is} some
  upper bound for the optimal SLDP value; $\mathcal{L}'$ is initially
  empty and will grow throughout the algorithm; this constraint will
  be used to guarantee that newly generated lot-types need to enter
  any solution with a total weight of at least one.
\item a reverse coupling constraint
  $y_l \le \sum_{m \in \mathcal{M}, b \in B}x_{\added[id=SK]{b,l,m}}$ for each lot-type
  $l \in \mathcal{L}$; this constraint ensures that a lot-type is
  only selected if it is assigned to some branch with some
  multiplicity, which will incur some cost~$c_{b,l,m}$.
\end{itemize}
Moreover, we denote by RSLDP (= restricted SLDP) the SLDP restricted
to some smaller set of lot-types, by MP (= master problem) the
linear programming relaxation of the SLDP including the set covering
constraint, by RMP (= restricted master problem) the MP restricted
to some smaller set of columns, and by PP (= pricing problem) the
pricing problem to determine whether there exist columns with
negative reduced costs. 

The restricted master problem (RMP) then reads as follows:
\begin{align}
  \min &&
          \sum_{b\in\mathcal{B}}\sum_{l\in\lottypesatbranch}\sum_{m\in\multsatbranchlottype}
          c_{b,l,m}\cdot x_{b,l,m} + {}&\added[id=JR]{\slackcost \slackvariable}&  && & (\text{duals:})\\
  s.t.  &&
           \label{ie_one_selection_per_branch}
           \sum_{l\in\lottypesatbranch}\sum_{m\in\multsatbranchlottype} x_{b,l,m} &= 1 && \forall b\in\mathcal{B} && (\alpha_b)\\
       &&
          -\sum_{m\in\multsatbranchlottype}x_{b,l,m}+y_l & \ge 0 && \forall b\in\mathcal{B}, l\in\lottypesatbranch && (\beta_{b,l})\\
       &&
          -\sum_{l\in\mathcal{L}'} y_l & \ge -k && && (\gamma)\\
       &&
          \added[id=JR]{\sum_{b \in \mathcal{B}: l \in \lottypesatbranch}\sum_{m \in \multsatbranchlottype}x_{b,l,m}-y_l} & \added[id=JR]{\ge 0} && \added[id=JR]{\forall l\in\mathcal{L}'} && \added[id=JR]{(\delta_l)}\\
       &&
       %% \label{ie_cover}
       %% \sum_{l\in\mathcal{C}_i} -y_l &\ge -\gamma_i && \forall
       %% i\in\mathcal{I}\\
  \label{ie_setcover}
  \added[id=JR]{\sum_{l \in \mathcal{L}'\setminus\mathcal{L}''}y_l + \slackvariable} &\added[id=JR]{\ge \added[id=SK]{1}} && && \added[id=JR]{(\mu)}\\
  % &&
  % -x_{b,l,m} & \ge -1 && \forall b\in\mathcal{B}, \forall
  % l\in\lottypesatbranch\forall m\in\multsatbranchlottype \\
  % &&
  % -y_l & \ge -1 && \forall l\in\mathcal{L}'\\
       && \!\!\!\!\!\!\!\!\!\!\!
          \sum_{b\in\mathcal{B}}\sum_{l\in\lottypesatbranch}\sum_{m\in\multsatbranchlottype}
          m \cdot \lvert l \rvert \cdot x_{b,l,m} &\ge \underline{I} && && \added[id=JR]{(\phi)}\\
       && \!\!\!\!\!\!\!\!\!\!\!
          -\sum_{b\in\mathcal{B}}\sum_{l\in\lottypesatbranch}\sum_{m\in\multsatbranchlottype}
          m \cdot \lvert l \rvert \cdot x_{b,l,m} &\ge -\overline{I} && && \added[id=JR]{(\psi)}\\
       &&
          x_{b,l,m} & \ge 0 && \forall b\in\mathcal{B},\notag\\
       &&&&& \forall l\in\lottypesatbranch,\notag\\
       &&&&& \forall m\in\multsatbranchlottype \\
       && y_l & \ge 0 && \forall l\in\mathcal{L}'.
       %% \end{align}
                         \intertext{\added[id=JR]{Using the indicated dual variables,} the dual restricted master
                         problem (DRMP) is then given by:}
                       %% \begin{align}
                         \max && \sum_{b\in\mathcal{B}}\alpha_b\,-k\gamma+\underline{I}\phi-\overline{I}\psi+\added[id=JR]{\mu}
                                                                                      && && & \added[id=JR]{(primals:)}\\
  s.t.  && \alpha_b-\beta_{b,l} \added[id=JR]{{} + \delta_l}+m\lvert l \rvert \phi-m\lvert l \rvert \psi &\le c_{b,l,m}
                                                                                && \forall b\in\mathcal{B},\notag\\
       &&&&&\forall l\in\lottypesatbranch,\notag\\
       &&&&&\forall m\in\multsatbranchlottype && \added[id=JR]{(x_{blm})}\label{dual_ie_1}\\
       && \sum_{b\in\mathcal{B}\,:\,l\in\lottypesatbranch}
          \beta_{b,l} \added[id=JR]{{} - \gamma - \delta_l} + \added[id=JR]{\mu} & \le
                                                                     0
                                                                                && \forall l\in\mathcal{L}' && \added[id=JR]{(y_l)}\label{dual_ie_2}\\
       &&
          \added[id=JR]{\added[id=JR]{\mu}} &\added[id=JR]{\le \slackcost} && && \added[id=JR]{(\slackvariable)}\label{dual_ie_3}\\
       &&
          \alpha_b & \in\mathbb{R} && \forall b\in\mathcal{B}\\
       &&
          \beta_{b,l} & \ge 0 && \forall b\in \mathcal{B}, l\in\lottypesatbranch\\
       &&
          \added[id=JR]{\delta_{l}} & \added[id=JR]{\ge 0} && \added[id=JR]{\forall
                                                l\in\mathcal{L}'}\\
       &&
          \gamma, \phi, \psi & \ge 0 && \\
       && \added[id=JR]{\mu} &\ge 0&& \deleted[id=JR]{\forall i\in\mathcal{I} \textbf{???}}.
\end{align}

In the following, we name all constraints by the names of the
variables dual to them.  That is, the DMP consists of
\replaced[id=JR, comment={Added $p$}]{$x$-, $y$-, and $\slackvariable$-constraints}{$x$-
  and $y$-constraints} whereas the MP consists of $\alpha$-, $\beta$-,
$\gamma$-, $\delta$-, $\mu$-, $\phi$-, and $\psi$-constraints.

The pricing problem (PP) is defined by finding those constraints in
the dual (DMP) of the unrestricted master problem (MP) that are most
violated in the DMP by the current solution of the DRMP.

Note that each feasible solution of the RMP induces a feasible
solution of the MP via a \emph{lifting} by zeroes in all missing
components.  Whenever the addition of variables to the RMP requires
the introduction of new constraints (like in our case), \added[id=JR,
comment={adapted this part to the previous changes}]{any exact
  formulation of} the \deleted[id=JR, comment={omit needless
  words}]{exact} pricing problem \deleted[id=JR, comment={removed
  obsolete reference}]{from Section~\ref{sec:modelling}} must be based
on a \emph{lifting} of the dual variables.  The following notions make
the formulation of our main result easier.
\begin{definition}
  Let the RMP contain \added[id=JR, comment={added $p$}]{the slack
    variable~$p$} and the variables and constraints for all lot-types
  $l \in \mathcal{L}'$ and assignments $(b, l, m)$ for all
  $b \in \mathcal{B}$,
  $l \in \lottypesatbranch \subseteq \mathcal{L}'$, and
  $m \in \multsatbranchlottype \subseteq \mathcal{M}$. At times we use
  for this the short-hand notation
  $(\mathcal{L}', L, M) := \bigl\{(b, l, m) : b \in \mathcal{B}, l \in
  \lottypesatbranch $\added[id=JR, comment={the lot-type set must
    appear somewhere}]{${} \subseteq \mathcal{L}'$},
  $m \in \multsatbranchlottype\}$, in which we consider $L$ and $M$ as
  set-valued functions. Moreover, let $(x, y\added[id=JR]{, p})$ be a
  feasible solution for the RMP and
  $(\alpha, \beta, \gamma, \delta, \mu, \phi, \psi)$ a feasible
  solution for the DRMP with identical objective value.  In
  particular, $(x, y\added[id=JR]{, p})$ and
  $(\alpha, \beta, \gamma, \delta, \mu, \phi, \psi)$ are optimal for
  the RMP and the DRMP, respectively.
  
  The \emph{canonical lifting} of
  $(x, y\added[id=JR]{, p})$ in the RMP is a
  complete set of primal variables
  $\bigl(\bar{x}, \bar{y}\added[id=JR]{, \bar{p}}\bigr)$ for all
  $b \in \mathcal{B}$, $l \in \mathcal{L}$, and $m \in \mathcal{M}$
  that arises from $(x, y\added[id=JR]{, p})$ by adding zeroes in the
  missing components.  Similarly, the \emph{canonical lifting} of the
  dual variables $(\alpha, \beta, \gamma, \delta, \mu, \phi, \psi)$ in
  the DRMP is a complete set of dual variables
  $(\bar{\alpha}, \bar{\beta}, \bar{\gamma}, \bar{\delta}, \bar{\mu},
  \bar{\phi}, \bar{\psi})$ for all $b \in \mathcal{B}$,
  $l \in \mathcal{L}$, and $m \in \mathcal{M}$ that arises from
  $(\alpha, \beta, \gamma, \delta, \mu, \phi, \psi)$ by adding zeroes
  in the missing components.
  
  A \emph{cost-invariant lifting} of the dual variables
  $(\alpha, \beta, \gamma, \delta, \mu, \phi, \psi)$ in the DRMP is a
  complete set of dual variables
  $(\hat{\alpha}, \hat{\beta}, \hat{\gamma}, \hat{\delta}, \hat{\mu},
  \hat{\phi}, \hat{\psi})$ with the following properties:
  \begin{enumerate}[(i)]
  \item The restriction of
    $(\hat{\alpha}, \hat{\beta}, \hat{\gamma}, \hat{\delta}, \hat{\mu},
    \hat{\phi}, \hat{\psi})$ to $(\mathcal{L}', \lottypesatbranchfunction, \multsatbranchlottypefunction)$ equals
    $(\alpha, \beta, \gamma, \delta, \mu, \phi, \psi)$.
  \item The objective in the DMP of
    $(\hat{\alpha}, \hat{\beta}, \hat{\gamma}, \hat{\delta}, \hat{\mu},
    \hat{\phi}, \hat{\psi})$ equals the objective in the DRMP of
    $(\alpha, \beta, \gamma, \delta, \mu, \phi, \psi)$.
  \end{enumerate}
  The \emph{(uncompensated) reduced cost $\bar{c}_{b, l, m}$} of a
  variable $x_{b, l, m}$ is defined as its reduced cost with respect
  to the canonical lifting, i.e.,
  \begin{equation}
    \bar{c}_{b, l, m} := c_{b, l, m} - \bar{\alpha}_b +
    \bar{\beta}_{b, l} - \bar{\delta}_l - m \abs{l}
    (\bar{\phi} - \bar{\psi}).
  \end{equation}
  Given a cost-invariant lifting
  $(\hat{\alpha}, \hat{\beta}, \hat{\gamma}, \hat{\delta}, \hat{\mu},
  \hat{\phi}, \hat{\psi})$, the \emph{compensated reduced cost
    $\hat{c}_{b, l, m}$} of a variable $x_{b, l, m}$ is defined as
  its reduced cost with respect to the cost-invariant lifting, i.e.,
  \begin{equation}
    \hat{c}_{b, l, m} := c_{b, l, m} - \hat{\alpha}_b +
    \hat{\beta}_{b, l} - \hat{\delta}_l - m \abs{l}
    (\hat{\phi} - \hat{\psi}).
  \end{equation}
  We call a cost-invariant lifting \emph{fully compensating}
  if all $x$-variables have non-negative compensated reduced costs.
\end{definition}
By weak duality, we have the following:
\begin{observation}
  Assume a cost-invariant lifting of dual variables for the DRMP
  is feasible for the DMP.  Then, the canonical lifting of an
  optimal solution to the RMP is optimal for the MP.  Moreover, for any
  cost-invariant lifting we have: A cost-invariant lifting
  satisfies all $x$-constraints if and only if it is fully
  compensating.
\end{observation}

The (PP) seeks for new variables with minimal reduced costs in the MP
and formally reads as follows for any cost-invariant lifting:
\begin{alignat}{4}
  \min
  \Bigl\{
  &\min \bigl\{ c_{b,l,m} - \hat{\alpha}_b + \hat{\beta}_{b,l} - \hat{\delta}_l - m\abs{l}(\hat{\phi}-\hat{\psi})
  &&:
  b \in \mathcal{B}, l \in \mathcal{L}, m \in \mathcal{M} \bigr\},\notag\\
  &\min \bigl\{  -\sum_{b \in
    \mathcal{B}}\hat{\beta}_{b,l} + \hat{\gamma} + \hat{\delta}_l - \hat{\mu}
  &&:
  l \in \mathcal{L} \bigr\}
  \Bigr\}.
\end{alignat}
Our idea is now to use a special cost-invariant lifting and check whether or
not it is fully compensating and feasible for the DMP.  We use the usual
notations $(x)^+ := \max \{x, 0\}$ and
$(x)^- := \max\{-x, 0\} = -\min \{x, 0\}$\added[id=JR,
comment={clarified the meaning for vectors}]{, componentwise}.
\begin{definition}
  For an optimal solution $(\alpha, \beta, \gamma, \delta, \mu, \phi, \psi)$
  of the DRMP we define the \emph{characteristic lifting} of
  dual variables by
  \begin{align}
    \check{\beta}_{b, l} &:=
                           \begin{cases}
                             \beta_{b, l} & \text{if $l \in \lottypesatbranch$},\\
                             \bigl(\min_{m \in \mathcal{M}} \bar{c}_{b,l,m} \bigr)^- & \text{if
                               $l \in \mathcal{L}' \setminus \lottypesatbranch \cup \mathcal{L}
                               \setminus \mathcal{L}'$},
                           \end{cases}\\
    \check{\delta}_l &:=
                       \begin{cases}
                         \delta_l & \text{if $l \in \mathcal{L}'$},\\
                         \bigl(\min_{\substack{b \in \mathcal{B}\\m \in \mathcal{M}}} \bar{c}_{b,l,m} \bigr)^+
                         & \text{if
                           $l \in \mathcal{L} \setminus \mathcal{L}'$}.
                       \end{cases}                       
  \end{align}
  Moreover, define as an abbreviation
  \begin{equation}
    \check{c}_l :=
    \begin{cases}
      -\sum_{\substack{b \in \mathcal{B}:\\l \in
          \mathcal{L}' \setminus \lottypesatbranch}}
      \check{\beta}_{b, l}
      & \text{if
        $l \in \mathcal{L}' \setminus \bigcap_{b \in \mathcal{B}} \lottypesatbranch$},\\
      -\sum_{b \in \mathcal{B}} \check{\beta}_{b, l}
      +
      \check{\delta}_l
      & \text{if
        $l \in \mathcal{L} \setminus \mathcal{L}'$}.
    \end{cases}        
  \end{equation}

\end{definition}
The importance of this lifting is demonstrated by the following theorem:
\begin{theorem}
  \label{thm:pricing}
  Consider a primal-dual pair of optimal solutions
  $(x, y\added[id=JR]{, p})$ of the RMP and
  $(\alpha, \beta, \gamma, \delta, \mu, \phi, \psi)$ of the DRMP.
  Then, there is a fully-compensating cost-invariant lifting feasible
  for the DMP if and only if the characteristic lifting is fully
  compensating and feasible for the DMP.

  In terms of uncompensated reduced costs this reads: If
  \begin{align}
    \bar{c}_{b,l,m} &\ge 0
    &&\forall b \in \mathcal{B},\notag\\
                    &&&l \in \lottypesatbranch,\notag\\
                    &&&m \in \mathcal{M} \setminus \multsatbranchlottype,\\
    -\sum_{\substack{b \in \mathcal{B}:\\l \in \mathcal{L}' \setminus \lottypesatbranch}}
    \bigl(
    \min_{m \in \mathcal{M}}
    \bar{c}_{b,l,m}
    \bigr)^-
                    &\ge \sum_{\substack{b \in \mathcal{B}:\\l \in \lottypesatbranch}}\beta_{b,l}
    - \gamma - \delta_l + \mu
                    && \forall l \in \mathcal{L}' \setminus \bigcap_{b \in \mathcal{B}} \lottypesatbranch,\\
    -\sum_{b \in \mathcal{B}}
    \bigl(
    \min_{m \in \mathcal{M}}
    \bar{c}_{b,l,m}
    \bigr)^-
    +
    \bigl(
    \min_{\substack{b \in \mathcal{B}\\m \in \mathcal{M}}}
    \bar{c}_{b,l,m}
    \bigr)^+
                    &\ge -\gamma + \mu
    &&\forall l \in \mathcal{L} \setminus \mathcal{L}',
  \end{align}
  then the canonical lifting
  $(\bar{x}, \bar{y}\added[id=JR]{, \bar{p}})$ of
  $(x, y\added[id=JR]{, p})$ is optimal for the MP.

  Moreover, if any of these inequalities is violated, then no cost-invariant
  lifting is fully-compensating and feasible for the DMP.
\end{theorem}
\begin{proof}
  Consider first the situation in which all inequalities listed in the
  theorem are satisfied, which is straight-forwardly equivalent to the
  characteristic lifting being fully compensating and feasible for the
  DMP. Moreover, by cost-invariance the characteristic lifting has the
  same objective as $(x, y\added[id=JR]{, p})$ and, therefore, the
  same objective as $(\bar{x}, \bar{y}\added[id=JR]{, \bar{p}})$.
  Thus, $(\bar{x}, \bar{y}\added[id=JR]{, \bar{p}})$ is optimal for
  the MP.

  \replaced[id=JR,comment={better line-break}]{Consider now}{Next
    consider} any fully compensating cost-invariant lifting
  $(\hat{\alpha}, \hat{\beta}, \hat{\gamma}, \hat{\delta}, \hat{\mu},
  \hat{\phi}, \hat{\psi})$ of
  $(\alpha, \beta, \gamma, \delta, \mu, \phi, \psi)$ of the DRMP that
  is feasible for the DMP.  We will show that then all conditions of
  the inequalities listed in the theorem are satisfied, i.e., the
  characteristic lifting is fully compensating and feasible for the
  DMP as well. By the definition of a lifting we know that
  $\hat{\alpha} = \alpha$, $\hat{\gamma} = \gamma$, $\hat{\mu} = \mu$,
  $\hat{\phi} = \phi$, and $\hat{\psi} = \psi$.  Since the lifting is
  fully compensating, the compensated reduced costs for all
  $x$-variables are non-negative.  In particular, the first inequality
  in the theorem is satisfied, since in that case
  $\hat{c}_{b, l, m} = \bar{c}_{b, l, m}$.

  Consider first the case $b \in \mathcal{B}$,
  $l \in \mathcal{L}' \setminus \lottypesatbranch$, and $m \in \mathcal{M}$.
  Now $\hat{\delta}_l = \delta_l$ must hold by
  the lifting property.
  We have
  \begin{equation}
    c_{b, l, m} - \alpha_b + \hat{\beta}_{b, l} - \hat{\delta}_l - m
    \abs{l} (\phi - \psi)
    = \bar{c}_{b, l, m} + \hat{\beta}_{b, l} \ge 0.
  \end{equation}
  From this, it follows that for all $b \in \mathcal{B}$ that
  \begin{equation}
    \hat{\beta}_{b, l} \ge \max_{m \in \mathcal{M}} (-\bar{c}_{b, l, m})
    = -\min_{m \in \mathcal{M}} \bar{c}_{b, l, m}.
  \end{equation}
  Since the given lifting is feasible for the DMP, this implies (using $\hat{\beta}_{b, l} \ge 0$)
  for all $l \in \mathcal{L}' \setminus \lottypesatbranch$:
  \begin{align}
    0
    &\ge \sum_{b \in \mathcal{B}} \hat{\beta}_{b, l} - \gamma -
      \hat{\delta}_l + \mu\notag\\
    &\ge
      \sum_{\substack{b \in \mathcal{B}:\\b \in \mathcal{L}' \setminus \lottypesatbranch}}
    \hat{\beta}_{b, l}
    +
    \sum_{\substack{b \in \mathcal{B}:\\b \in \lottypesatbranch}}
    \beta_{b, l}
    - \gamma - \delta_l + \mu\notag\\
    &=
      \sum_{\substack{b \in \mathcal{B}:\\b \in \mathcal{L}' \setminus \lottypesatbranch}}
    \max\bigl\{0, \hat{\beta}_{b, l}\bigr\}
    + \sum_{\substack{b \in \mathcal{B}:\\b \in \lottypesatbranch}}
    \beta_{b, l}
    - \gamma - \delta_l + \mu\notag\\
    &\ge
      \sum_{\substack{b \in \mathcal{B}:\\b \in \mathcal{L}' \setminus \lottypesatbranch}}
    \max \bigl\{0, -\min_{m \in \mathcal{M}} \bar{c}_{b, l, m}\bigr\}
    + \sum_{\substack{b \in \mathcal{B}:\\b \in \lottypesatbranch}}
    \beta_{b, l}
    - \gamma - \delta_l + \mu\notag\\
    &\ge
      \sum_{\substack{b \in \mathcal{B}:\\b \in \mathcal{L}' \setminus \lottypesatbranch}}
    \bigl(\min_{m \in \mathcal{M}} \bar{c}_{b, l, m}\bigr)^-
    + \sum_{\substack{b \in \mathcal{B}:\\b \in \lottypesatbranch}}
    \beta_{b, l}
    - \gamma - \delta_l + \mu.
  \end{align}
  This is equivalent to the second inequality in the theorem.  Consider next
  the case $b \in \mathcal{B}$, $l \in \mathcal{L} \setminus \mathcal{L}'$,
  and $m \in \mathcal{M}$.  We have
  \begin{equation}
    c_{b, l, m} - \alpha_b + \hat{\beta}_{b, l} - \hat{\delta}_l - m
    \abs{l} (\phi - \psi)
    = \bar{c}_{b, l, m} + \hat{\beta}_{b, l} - \hat{\delta}_l \ge 0.
  \end{equation}
  Let $(b^*, m^*)$ be a minimizer in
  $\min_{\substack{b \in \mathcal{B}, m \in \mathcal{M}}}\bar{c}_{b,
    l, m} = \bar{c}_{b^*, l, m^*}$.  \replaced[id=JR,comment={better
    line-break}]{That is}{In particular},
  $\bar{c}_{b^*, l, m^*} \le \min_{m \in \mathcal{M}} \bar{c}_{b, l,
    m}$ for all $b \in \mathcal{B}$.  We consider the cases
  $\bar{c}_{b^*, l, m^*} \le 0$ and $\bar{c}_{b^*, l, m^*} > 0$.  In
  the first case, since the lifting is fully compensating, we have
  $\hat{\beta}_{b, l} - \hat{\delta}_l \ge -\min_{m \in
    \mathcal{M}}\bar{c}_{b, l, m}$ for all $b \in \mathcal{B}$, in
  particular
  \begin{align}
    \hat{\beta}_{b^*, l} - \hat{\delta}_l
    &\ge
      -c_{b^*, l, m^*} \ge 0,\\
    \hat{\beta}_{b, l}
    &\ge
      -\min_{m \in \mathcal{M}}\bar{c}_{b, l, m}.
  \end{align}
  We conclude for all $l \in \mathcal{L} \setminus \mathcal{L}'$, using
  $\hat{\delta}_l \ge 0$, $\hat{\beta}_{b, l} \ge 0$:
  \begin{align}
    0
    &\ge \sum_{b \in \mathcal{B}} \hat{\beta}_{b, l} - \gamma - \hat{\delta}_l + \mu\notag\\
    &\ge \sum_{b \in \mathcal{B} \setminus \{b^*\}} \hat{\beta}_{b, l} 
      + \bigl(\hat{\beta}_{b^*, l}  - \hat{\delta}_l\bigr)
      - \gamma + \mu\notag\\
    &= \sum_{b \in \mathcal{B} \setminus \{b^*\}}
      \max\bigl\{0, \hat{\beta}_{b, l}\bigr\}
      + \max\{0, \hat{\beta}_{b^*, l}  - \hat{\delta}_l\}
      - \gamma + \mu\notag\\
    &\ge \sum_{b \in \mathcal{B} \setminus \{b^*\}}
      \max\bigl\{0, -\min_{m \in \mathcal{M}} \bar{c}_{b, l, m}\bigr\}
      + \max\bigl\{0, -\min_{m \in \mathcal{M}} \bar{c}_{b^*, l, m}\bigr\}
      - \gamma + \mu\notag\\
    &= \sum_{b \in \mathcal{B}} \max\bigl\{0, -\min_{m \in \mathcal{M}} \bar{c}_{b, l, m}\bigr\} - \gamma + \mu\notag\\
    &= \sum_{b \in \mathcal{B}} \bigl(\min_{m \in \mathcal{M}} \bar{c}_{b, l, m}\bigr)^- - \gamma + \mu.
  \end{align}
  This is equivalent to the third inequality of the theorem in the case
  $\bar{c}_{b^*, l, m^*} \le 0$.  In the case $\bar{c}_{b^*, l, m^*} > 0$, we use
  that
  \begin{align}
    0 \le \hat{\delta}_l
    &\le
      \min_{\substack{b \in \mathcal{B}\\m \in \mathcal{M}}}
    \bigl(\bar{c}_{b, l, m}
    + \hat{\beta}_{b, l}
    \bigr)
    =
    \bar{c}_{b^*, l, m^*} + \hat{\beta}_{b^*, l},\\
    \hat{\beta}_{b, l}
    &\ge
      -\min_{m \in \mathcal{M}}\bar{c}_{b, l, m}.
  \end{align}
  Then, again using $\hat{\beta}_{b, l} \ge 0$ and
  $\min_{m \in \mathcal{M}} \bar{c}_{b, l, m} \ge \bar{c}_{b^*, l,
    m^*} > 0$, we get:
  \begin{align}
    0
    &\ge \sum_{b \in \mathcal{B}} \hat{\beta}_{b, l} - \gamma - \hat{\delta}_l + \mu\notag\\
    &\ge \sum_{b \in \mathcal{B}} \hat{\beta}_{b, l} - \bigl(\bar{c}_{b^*, l,
      m^*} + \hat{\beta}_{b^*, l}\bigr) - \gamma  + \mu\notag\\
    &= \sum_{b \in \mathcal{B} \setminus \{b^*\}} \max \bigl\{0, \hat{\beta}_{b, l}\bigr\}
      - \bar{c}_{b^*, l, m^*} - \gamma + \mu\notag\\ 
    &\ge  \sum_{b \in \mathcal{B} \setminus \{b^*\}} \max \bigl\{0, -\min_{m \in \mathcal{M}}
      \bar{c}_{b, l, m}\bigr\} - \bar{c}_{b^*, l, m^*} - \gamma + \mu\notag\\ 
    &=  -\bigl(\bar{c}_{b^*, l, m^*} \bigr)^+ -\gamma + \mu.
  \end{align}
  This is equivalent to the third inequality of the theorem in the case
  $\bar{c}_{b^*, l, m^*} > 0$.  Together, these two cases imply the third
  inequality in the theorem.  Summarized, the characteristic lifting is feasible as
  well, as claimed.
\end{proof}
Clearly, we will generate new columns with respect to the compensated reduced
costs with respect to the characteristic lifting.  Then, only complete sets of
columns will be generated whose negative reduced costs cannot be fully
compensated by \emph{any} lifting.  We call such a variable set
\emph{promising}.

The individual $x_{b, l, m}$ with $\bar{c}_{b, l, m} < 0$ for which
$l \in \lottypesatbranch$ and $m \in \mathcal{M} \setminus \multsatbranchlottype$ form
single-element promising sets of variables because for those we have
$\hat{c}_{b, l, m} = \bar{c}_{b l, m}$, i.e., no lifting can compensate their
negative uncompensated reduced costs.

Another promising set of variables stems from any fixed
$l \in \mathcal{L}' \setminus \bigcap_{b \in \mathcal{B}} \lottypesatbranch$ for which
\begin{equation}
  \check{c}_l = -\sum_{\substack{b \in \mathcal{B}:\\l \in \mathcal{L}' \setminus \lottypesatbranch}}
  \bigl( \min_{m \in \mathcal{M}} \bar{c}_{b,l,m} \bigr)^- < \sum_{\substack{b
      \in \mathcal{B}:\\l \in \lottypesatbranch}}\beta_{b,l} - \gamma - \delta_l + \mu.
\end{equation}
The set contains all $x_{b, l, m}$ with
$\min_{m \in \mathcal{M}} \bar{c}_{b,l,m} < 0$.

And finally, we have a promising set of variables from any
$l \in \mathcal{L} \setminus \mathcal{L}'$ for which
\begin{equation}
  \check{c}_l = -\sum_{b \in \mathcal{B}} \bigl( \min_{m \in \mathcal{M}} \bar{c}_{b,l,m}
  \bigr)^- + \bigl( \min_{\substack{b \in \mathcal{B}\\m \in \mathcal{M}}}
  \bar{c}_{b,l,m} \bigr)^+ < -\gamma + \mu.
\end{equation}
The set contains all $x_{b, l, m}$ with
$\min_{m \in \mathcal{M}} \bar{c}_{b,l,m} < 0$ plus the new variable $y_l$.

Note, that the resulting mechanism is very natural, since it essentially
enforces a \emph{coordinated} generation of new variables taking the quality
of an assignment into account with respect to all branches at the same time.

We complete our consideration by two lower bounds implied by using
essentially two non-cost-invariant fully compensating \replaced[id=JR,
comment={grammar correction}]{shifts}{shift} of dual variables (that
is no lifting, by the way) that \replaced[id=JR, comment={grammar
  correction}]{are}{is} feasible for the DMP.

The first idea is to shift the $\alpha$-variables of the DMP so as
to compensate all the negative reduced costs of variables
$x_{b, l, m}$ with $b \in \mathcal{B}$, $l \in \lottypesatbranch$,
and $m \in \mathcal{M} \setminus \multsatbranchlottype$.  The result
is that for all such $(b, l, m)$ the $x$-constraints of the DMP are
feasible.  The second idea is to shift the $\gamma$-variable of the
DMP so as to compensate the violation of the $y$-constraints, while
the reduced costs of the variables $x_{b, l, m}$ for
$b \in \mathcal{B}$,
$l \in \mathcal{L} \setminus \lottypesatbranch \cup \mathcal{L}
\setminus \mathcal{L}'$, and $m \in \mathcal{M}$ are fully
compensated by the characteristic lifting.

More formally, define the minimal negative reduced costs of any
promising variable $x_{b, l, m}$ with $b \in \mathcal{B}$,
$l \in \lottypesatbranch$, and
$m \in \mathcal{M} \setminus \multsatbranchlottype$ as
\begin{equation}
  \bar{c}^*_b := -\bigl(\min_{\substack{l \in \lottypesatbranch\\m \in
      \multsatbranchlottype}} \bar{c}_{b, l, m} \bigr)^-
  \le 0.
\end{equation}
Next, define the maximal violation of the $y$-constraint in the DMP
by the characteristic lifting as
\begin{equation}
  \bar{d}^* := -\bigl( \min_{\substack{l \in \mathcal{L}' \setminus
      \bigcap_{b \in \mathcal{B}}\lottypesatbranch \cup \mathcal{L} \setminus \mathcal{L}'\\
      m \in \multsatbranchlottype}}
  -\sum_{b \in \mathcal{B}} \check{\beta}_{b, l}
  + \check{\delta}_l
  + \gamma - \mu
  \bigr)^- \le 0.
\end{equation}

Next, we define a shift of the characteristic lifting by
\begin{align}
  \tilde{\alpha}_b     &:= \alpha_b + \bar{c}^*_b,\\
  \tilde{\beta}_{b, l} &:= \check{\beta}_{b, l},\\
  \tilde{\gamma}      &:= \gamma - \bar{d}^*\\
  \tilde{\delta}_{l}  &:= \check{\delta}_{l},\\
  \tilde{\mu}         &:= \mu,\\
  \tilde{\phi}         &:= \phi,\\
  \tilde{\psi}         &:= \psi.
\end{align}
We call this the \emph{lower-bound shift of the characteristic
  lifting}.
\begin{theorem}
  \label{thm:bound}
  The lower-bound shift of the characteristic lifting is feasible
  for the DMP.  Moreover, it changes the objective function value
  of the characteristic lifting by $\sum_{b \in \mathcal{B}}
  \bar{c}^*_b + k\bar{d}^*$.  In other words, if $z^{\mathit{RMP}}$
  is the optimal value of the RMP and $z^{\textit{MP}}$ is the
  optimal value of the MP, then we have
  \begin{equation}
    z^{\textit{MP}} \ge z^{\textit{CSB}} :=
    z^{\textit{RMP}} + \sum_{b \in \mathcal{B}} \bar{c}^*_b + k\bar{d}^*.
  \end{equation}
\end{theorem}
\begin{proof}
  Since $\bar{d}^* \le 0$ we have that $\tilde{\gamma}$ is
  non-negative.  Note, that there is no non-negativity constraint
  for~$\alpha_b$.  We consider the $x$-constraints in two steps:
  \begin{enumerate}
  \item The $x$-constraints with $b \in \mathcal{B}$,
    $l \in \lottypesatbranch$, and $m \in \mathcal{M}$ read:
    \begin{align}
      \tilde{\alpha}_b - \tilde{\beta}_{b,l} + \tilde{\delta}_l+ m\abs{l}
      (\tilde{\phi} - \tilde{\psi})
      &\le c_{b,l,m}\notag\\
      \iff
      \alpha_b + \bar{c}^*_b - \beta_{b,l} + \delta_l
      +
      m\abs{l}
      (\phi - \psi)
      &\le c_{b,l,m}\\
      \iff
      \bar{c}_{b, l, m} = 
      c_{b, l, m} - \alpha_b + \beta_{b,l} - \delta_l
      -
      m\abs{l}
      (\phi - \psi)
      &\ge \bar{c}^*_b,
    \end{align}
    which holds by the minimality of $\bar{c}^*_b$.
  \item The $x$-constraints for all $(b, l, m)$ with
    $l \in \mathcal{L}' \setminus \lottypesatbranch \cup \mathcal{L}
    \setminus \mathcal{L}'$ are satisfied
    since the characteristic lifting is fully compensating, and its
    lower-bound shift relaxes the $x$-constraints.
  \end{enumerate}
  Next we consider the $y$-constraints.
  \begin{enumerate}
  \item The $y$-constraints with
    $l \in \bigcap_{b \in \mathcal{B}}\lottypesatbranch$ are satisfied
    by the \added[id=SK]{characteristic} lifting, since they are identical to
    those in the DRMP, and there they are satisfied by the lifting
    property.  The shift of $\gamma$ relaxes the $y$-constraints, so
    that also the lower-bound shift satisfies these $y$-constraints.
  \item The $y$-constraints with
    $l \in \mathcal{L}' \setminus \bigcap_{b \in
      \mathcal{B}}\lottypesatbranch \cup \mathcal{L} \setminus
    \mathcal{L}'$ read:
    \begin{align}
      \sum_{\substack{b \in \mathcal{B}}}
      \tilde{\beta}_{b,l} - \tilde{\gamma} - \tilde{\delta}_l + \tilde{\mu}
      &\le 0,\\
      \iff
      \sum_{\substack{b \in \mathcal{B}}}
      \check{\beta}_{b,l} - (\gamma - \bar{d}^*) - \check{\delta}_l + \mu
      &\le 0,\\
      \iff
      -\sum_{\substack{b \in \mathcal{B}}}
      \check{\beta}_{b,l} - \gamma - \check{\delta}_l + \mu
      &\ge \bar{d}^*,
    \end{align}
    which holds by the minimality of $\bar{d}^*$.
  \end{enumerate}
  Thus, the lower-bound shift of the characteristic lifting is
  feasible for the DMP.  Plugging the lower-bound shift into the
  objective function of the DMP yields the lower-bound formula.
\end{proof}
We will refer to this lower bound by the name \emph{characteristic
  shift bound}.  There may be tighter bounds by other shifts.  The
advantage of the characteristic shift bound is that is is readily
available in any algorithm that computes the promising sets of
variables.

\section{A branch-and-price algorithm for the SLDP}
\label{sec_column_generation}

\noindent
\deleted[id=JR, comment={does not contribute to the understanding}]{On
  smaller instances, there is evidence that the \emph{score-fix-adjust
    heuristic (SFA)} yields close-to-optimal
  solutions~\cite{p_median}.  However, it is of interest whether the
  performance of SFA is satisfying on larger instances, too.}

Since the set of applicable lot-types and, thus, the set of binary
variables in the stated ILP formulation may become too large for a
static ILP solution, a natural approach is to consider applicable
lot-types and the corresponding assignment options dynamically
in a branch-and-price algorithm.

In this section we show how special structure can be used to obtain a
fast branch-and-price algorithm for practically relevant instances: We
typically have $300\le|\mathcal{B}|\le 1600$ and $3\le|\mathcal{M}|\le
7$ while $\lvert\mathcal{L}\rvert$ can be around $10^9$, see the example
stated in the introduction.

\subsection{The top-level algorithm ASG}
\label{sec:ASG}

The idea of our exact branch-and-price algorithm is based
on the following practical observations on real-world data:
\begin{itemize}
\item The integrality gap of our SLDP model is small.
\item Solutions generated by heuristics perform very well (see
  \cite{p_median} for SFA).
\item An MILP solver can solve instances of the ILP model with few
  lot-types very quickly.
\item There seems to be a ``small'' set of \added[id=SK]{{\lq\lq}good{\rq\rq}} and a
  ``large'' set of \added[id=SK]{{\lq\lq}bad{\rq\rq}} lot-types.
\item No mathematical structure of the set of \added[id=SK]{{\lq\lq}good{\rq\rq}} lot-types
  is known a-priori.
\end{itemize}

We want to exploit the fast heuristic and the tight ILP model.
Instead of using only the LP-solver part of an MILP solver and branch
to integrality manually in our own branch-and-price tree, we decided
to utilize as much of the modern ILP solver technology as possible in
our algorithm by generating and solving a sequence of ILP models
handling distinct subproblems with only few lot-types.  We call this
branch-and-price method \emph{augmented subproblem grinding (ASG)}.
\added[id=JR, comment={motivation for the name ``ASG''}]{The name is
  motivated by the interpretation that in an iteration we
  ``grind-off'' a promising subproblem from the full remaining
  subproblem, solve it to optimality, and proceed with the remaining
  subproblem that does not contain the solved promising subproblem
  anymore.} The implementation of this method needs a much smaller
effort than a typical branch-and-price implementation with branching,
e.g., on the variables of a compact model formulation like the one in
appendix~\ref{sec:comp-model-param}.

With this, the outline of ASG in its simplest version reads as follows.
\begin{enumerate}
\item Run the SFA heuristics to obtain an incumbent feasible SLDP
  solution $(x^*, y^*\added[id=JR]{, p^*})$ that
  implies an upper bound~$z^{\mathit{upper}}$ for the SLDP optimum.
\item Generate a small set of initial columns for the RMP so that
  the RMP is feasible. Solve the RMP. Enter the pricing phase.
\item The pricing phase: While the PP returns new columns with
  negative reduced costs, do the following:
  \begin{enumerate}
  \item Add (some of) these columns to the RMP and resolve.  
  \item From the PP update~$z^{\mathit{lower}}$ using the
    characteristic shift bound.
  \item If $z^{\mathit{lower}} \ge z^{\mathit{upper}}$, return
    $(x^*, y^*\added[id=JR]{, p^*})$ and their
    cost~$z^{\mathit{upper}}$.
  \end{enumerate}
  At the end of this step, we obtain an optimal MP solution which
  implies an updated lower bound $z^{\mathit{lower}}$ for the SLDP
  optimum.  If $z^{\mathit{lower}} \ge z^{\mathit{upper}}$, return
  $(x^*, y^*\added[id=JR]{, p^*})$ and their
  cost~$z^{\mathit{upper}}$.  Else enter the cutting phase.
\item The cutting phase: Add all lot-types
  $\mathcal{L}' \subset \mathcal{L}$ from the RMP to the RSLDP and
  solve the RSLDP. \added[id=JR, comment={explained what this has to
    do with branching and named the branch}]{This is the
    \emph{promising-subproblem branch} of ASG, which is solved
    immediately.}\footnote{In various applications with tight models,
    the procedure is stopped at this point because one hopes that the
    integrality gap is so small that further effort need not be spent
    (see, e.g., \cite{adac}).  ASG can be seen as extending this into
    an exact algorithm finding the optimum ILP solution.}  If the
  RSLDP solution is cheaper than the current incumbent
  $(x^*, y^*\added[id=JR]{, p^*})$, update the
  incumbent $(x^*, y^*\added[id=JR]{, p^*})$ and
  the upper bound~$z^{\mathit{upper}}$. If
  $z^{\mathit{lower}} \ge z^{\mathit{upper}}$, return
  $(x^*, y^*\added[id=JR]{, p^*})$ and their
  cost~$z^{\mathit{upper}}$.  Otherwise, strengthen the set covering
  constraint in the SLDP to
  $\sum_{l \in \mathcal{L} \setminus \mathcal{L}'} y_l +
  \replaced[id=JR]{\slackvariable}{s} \ge 1$
  yielding $\replaced[id=JR]{\slackvariable}{s} \ge 1$
  in the RMP, solve the resulting RMP, and enter the pricing
  phase. \added[id=JR, comment={named the branch}]{This is the
    \emph{remaining-subproblem branch} of ASG.}
\end{enumerate}

The top-level algorithm for ASG is listed in pseudo code in
Algorithm~\ref{alg:top}.  Several subroutines are used that will be
explained in the following sections.

\SetKwData{RMPTotalCols}{\ensuremath{(\mathcal{L}', \lottypesatbranchfunction, \multsatbranchlottypefunction)}}
\SetKwData{RMPUpdatedCols}{\ensuremath{
    (\mathcal{L}' \cup \mathcal{L}^{\mathrm{new}},
    \lottypesatbranchfunction \cup \lottypesatbranchfunction^{\mathrm{new}},
    \multsatbranchlottypefunction \cup \multsatbranchlottypefunction^{\mathrm{new}})}}
\SetKwData{RMPNewCols}{\ensuremath{(\mathcal{L}^{\mathrm{new}},
    \lottypesatbranchfunction^{\mathrm{new}}, \multsatbranchlottypefunction^{\mathrm{new}})}}
\SetKwData{UpperBound}{\ensuremath{z^{\mathit{upper}}}}
\SetKwData{LowerBound}{\ensuremath{z^{\mathit{lower}}}}
\SetKwData{Incumbent}{\ensuremath{(x^*, y^*, p^*)}}
\SetKwData{LiftedIncumbent}{\ensuremath{(\bar{x}^*, \bar{y}^*, \bar{p}^*)}}
\SetKwData{RMPSetCoverSupport}{SetCoverSupport}
\SetKwFunction{GLB}{GenerateLocBestLottypes}
\SetKwData{ScoreTable}{ScoreTable}
\SetKwFunction{SFA}{ScoreFixAdjust$^+$}
\SetKwFunction{RMPinit}{InitRMP}
\SetKwFunction{RMPaddcols}{RMPaddcols}
\SetKwFunction{RMP}{SolveRMP}
\SetKwFunction{RSLDP}{SolveRSLDP}
\SetKwFunction{PP}{SolvePP}
\SetKwFunction{ASG}{ASG}
\begin{algorithm}[htbp]
  \KwIn{A consistent SLDP; $K_1, K_2, K_3, K_4, K_5 \in \mathbb{N}$ (Parameters)}
  \KwOut{An optimal solution for the SLDP}
  $\UpperBound \longleftarrow \infty$\;
  $\LowerBound \longleftarrow -\infty$\;
  $(\mathcal{L}', \ScoreTable) \longleftarrow \GLB{$K_1, K_2$}$\;
  $\mathcal{L}'' \longleftarrow \emptyset$\;
  $\bigl(\LiftedIncumbent, \UpperBound\bigr) \longleftarrow
  \SFA{$\ScoreTable$}$\;
  $\RMPTotalCols \longleftarrow \RMPinit{$K_1, K_2, K_3$}$\;
  \While{true}{
    \Repeat{$\RMPNewCols = (\emptyset, \emptyset, \emptyset)$}{
      \tcp{The pricing phase:}
      $(\alpha, \beta, \gamma, \delta, \mu, \phi, \psi; z^{\mathit{RMP}}) \longleftarrow \RMP{$\RMPTotalCols; \mathcal{L}''$}$\;
      $\bigl(\RMPNewCols, \LowerBound\bigr) \longleftarrow
      \PP{$\alpha, \beta, \gamma, \delta, \mu, \phi, \psi;
        z^{\mathit{RMP}}; \mathcal{L}'; \LowerBound; K_4, K_5$}$\;
      \uIf{$\UpperBound \le \LowerBound$}{
        \Return{$\LiftedIncumbent$}\;
      }
      \Else{
        $\RMPTotalCols \longleftarrow \RMPUpdatedCols$\;
      }
    }
    \tcp{The cutting phase:}
    $\bigl(\Incumbent, \UpperBound\bigr) \longleftarrow \RSLDP{$\mathcal{L}'$}$\;
    \uIf{$\UpperBound \le \LowerBound$}{\Return{$\LiftedIncumbent$}}
    \Else{
      $\mathcal{L}'' \longleftarrow \mathcal{L}'$\;
    }
  }
  \caption{Top Level Algorithm \protect\ASG for an SLDP}\label{alg:top}
\end{algorithm}

\subsection{The subroutine \protect\GLB{$K,K'$}}

We call an element
$(b, l, m^*) \in \mathcal{B} \times \mathcal{L} \times \mathcal{M}$ a
\emph{locally optimal multiplicity assignment}, if
$c_{b, l, m^*} = \min_{m \in \mathcal{M}} c_{b,l,m}$.  (Note that
$m^*$ depends on $b$ and~$l$.)  Moreover, we call an element
$(b, l^*, m^*) \in \mathcal{B} \times \mathcal{L} \times \mathcal{M}$
a \emph{locally optimal lot-type-multiplicity assignment}, if
$c_{b, l^*, m^*} = \min_{l \in \mathcal{L}} c_{b,l,m^*}$.  For a
branch $b \in \mathcal{B}$ and a number
$K \in \mathbb{N}$ we call a list of lot-types
$\bigl(l_1(b), \dots, l_K(b)\bigr)$ a \emph{list of $K$ locally best
  lot-types for branch~$b$}, if for all $j = 1, \dots, K$ there are at
most $j - 1$ many $l \in \mathcal{L}$ with
$c_{b, l, m^*} < c_{b, l_j, m^*}$.

For parameters~$K, K' \in \mathbb{N}$, the subroutine
\ensuremath{\mathrm{GenerateLocBestLottypes}(K,K')}
generates a subset of lot-types in the following way: first, for each
branch $b \in \mathcal{B}$ a list $\bigl(l_1(b), \dots, l_K(b)\bigr)$
of $K$ locally best lot-types is determined.  Then, we score the
lot-types by adding a score of $-\bigl(10^{K-j}\bigr)$ to the $j$th
best lot-type for branch~$b$, $j = 1, \dots, K$.  From the score
ranking we pick the $K'$ elements with the smallest scores.  The pseudo
code can be found in Algorithm~\ref{alg:GLB}.

\SetKwData{pLottype}{PartialLottype}
\SetKwFunction{FLB}{FindLocBestLottypes}
\begin{algorithm}[htbp]
  \KwIn{A consistent SLDP and $K, K' \in \mathbb{N}$}
  \KwOut{A subset of lot-types $\mathcal{L}' \subseteq \mathcal{L}$
    and a sorted table $\ScoreTable$ of all lot-types with non-zero score}
  
  \tcp{store lot-types (keys) with scores (values) in a table,
    initially empty:}
  $\ScoreTable \longleftarrow \bigl(\bigr)$\;
  \For{$b \in \mathcal{B}$}{
    $\bigl(l_1(b), \dots, l_K(b)\bigr) \longleftarrow \FLB{$b, K$}$\;
    \For{$j$ from $1$ to $K$}{
      \If{$\ScoreTable$ does not yet contain $l_j$}{
        insert $l_j$ into $\ScoreTable$ with value $0$\;
      }
      $\ScoreTable(l_j) \longleftarrow \ScoreTable(l_j) - 10^{K-j}$
    }
  }
  $\mathcal{L}' \longleftarrow \{$ the best $K'$ lot-types in \ScoreTable$ \}$\;
  \Return{$(\mathcal{L}', \ScoreTable)$}\;
  \caption{\protect\GLB}\label{alg:GLB}
\end{algorithm}

\SetKwData{pLottype}{\ensuremath{l'}}
\SetKwData{LTList}{LotTypeList}
\SetKwFunction{RecFind}{RecLocBestLottypes}
\begin{algorithm}[htbp]
  \KwIn{A consistent SLDP, $b \in \mathcal{B}$, $K \in \mathbb{N}$}
  \KwOut{A list $\LTList = \bigl( l_1, \dots, l_K \bigr)$ of $K$ locally best lot-types for branch~$b$}
  
  \tcp{start with an empty sorted list:}
  $\LTList \longleftarrow \bigl(\bigr)$\;
  \tcp{start with an empty partial lot-type:}
  $l' \longleftarrow \bigl(\bigr)$\;
  \tcp{applicable range of cardinalities for the first size:}
  $n_{\min} \longleftarrow \min_t - (\abs{S} - 1) \max_c$\;
  $n_{\max} \longleftarrow \max_t - (\abs{S} - 1) \min_c$\;
  $\RecFind{$b, K, l', n_{\min}, n_{\max}; \LTList$}$\;
  \Return{$\LTList$}\;
  \caption{\protect\FLB}\label{alg:FLB}
\end{algorithm}
\SetKwData{CList}{CardList}
\SetKwFunction{BBLowerBound}{\ensuremath{\lambda_b(l'')}}
\begin{algorithm}[htbp]
  \KwIn{A consistent SLDP, $b \in \mathcal{B}$, $K \in \mathbb{N}$, a
    partial lot-type~$l'$, $n_{\min}, n_{\max} \in \mathbb{N}$, a
    list $\LTList = \bigl( l_1, \dots, l_K \bigr)$ of lot-types found so far}
  \KwOut{An updated list $\LTList$ of lot-types}
  
  \uIf{$l'$ is complete}{
    \tcp{lot-type reached - check cost and update lot-type list:}
    \If{$\abs{\LTList} < K$ or $c_{b, l', m^*} < c_{b, l_K, m^*}$}{
      insert $l'$ into $\LTList$\;
    }
  }
  \Else{ 
    \tcp{list extensions of $l'$ sorted by lower bounds:}
    $\CList \longleftarrow \bigl(\bigr)$\;
    \For{$n$ from $n_{\min}$ to $n_{\max}$}{
      $l'' \longleftarrow \bigl(l', n \bigr)$\;
      \If{$\BBLowerBound \le c_{b, l_K, m^*}$}{
        \tcp{lower bound does not exclude $l''$:}
        insert $n$ into $\CList$\;
      }
    }
    sort $\CList$ by increasing $\BBLowerBound$\;
    \For{$n$ in $\CList$}{
      \tcp{process all non-pruned extensions:}
      $l'' \longleftarrow \bigl(l', n \bigr)$\;
      $n'_{\min} \longleftarrow n_{\min} + (\max_c - n)$\;
      $n'_{\max} \longleftarrow n_{\max} - (n - \min_c)$\;
      $\RecFind{$b, K, l'', n'_{\min}, n'_{\max}; \LTList$}$\;
    }
  }
  \caption{\protect\RecFind}
  \label{alg:recFLB}
\end{algorithm}

Since determining the $K$ locally best lot-types for a branch requires
a search in the complete, possibly very large lot-type set
$\mathcal{L}$, we rather synthesize lot-types by determining the
number of pieces for each size separately in a branch-and-bound
algorithm named \ensuremath{\mathrm{FindLocBestLottypes}(b, K)}
that recursively extends partial lot-types in a depth-first-search
manner.  In each branch-and-bound node, the cost for \added[id=SK]{the completion} 
of a partial lot-type is bounded from below by the minimal cost over
all possible multiplicities of the cheapest (not necessarily feasible)
completions.  Each leaf in the branch-and-bound tree yields an
applicable lot-type and, together with its locally optimal
multiplicity assignment, a cost.  We maintain a sorted list of the
$K$ best lot-types. As long as the list contains less than $K$
elements, each new (complete) applicable lot-type is added to it.  If
the list is full, the cost of each new lot-type is compared to the
worst in the list.  If it is cheaper, then it is exchanged with the
current worst solution, and the list is resorted.  Moreover, once the
list is full, the lower bound of each partial lot-type is compared to
the worst in the list.  If the lower bound exceeds the cost of the
worst element in the list, then the node corresponding to the partial
lot-type is pruned.  In the following, we explain the procedure 
more formally.

A \emph{partial lot-type} is a vector
$(l'_{s'})_{s' \in \mathcal{S}'}\in\mathbb{N}^{\mathcal{S}'}$ for
some $\mathcal{S}' \subseteq \mathcal{S}$.  As for complete
lot-types, the number of pieces in this partial lot-type is denoted
by $\abs{l'} := \sum_{s' \in \mathcal{S}'} l'_{s'}$.  A
\emph{completion} of a partial lot-type $l'$ is a lot-type $l$ with
$l_{s'} = l'_{s'}$ for all $s' \in \mathcal{S}'$.  For any partial
lot-type $l'$ we consider the two extremal completions
$\underline{l'}$ with
$(\underline{l'})_{s \in \mathcal{S} \setminus \mathcal{S}'} :=
(\min_c)_{s \in \mathcal{S} \setminus \mathcal{S}'}$ and
$\overline{l'}$ with
$(\overline{l'})_{s \in \mathcal{S} \setminus \mathcal{S}'} :=
(\max_c)_{s \in \mathcal{S} \setminus \mathcal{S}'}$, where we fix
the numbers of pieces for missing sizes to the minimal and maximal
applicable values, respectively.  Consequently, the minimal and
maximal total numbers of pieces for the missing sizes are given by
$\abs{\underline{l'}}$ and $\abs{\overline{l'}}$, respectively.

A partial lot-type~$l'$ is \emph{applicable} if
$\min_c \le l'_{s'} \le \max_c$ for all $s' \in \mathcal{S}'$, and,
moreover, $\abs{\overline{l'}} \ge \min_t$, and
$\abs{\underline{l'}} \le \max_t$, i.e., there exists \replaced[id=JR,
comment={minor rephrasing}]{a}{an applicable} completion to an
applicable lot-type.  For a given branch~$b$, the \emph{partial cost}
of a partial lot-type~$l'$ if delivered with multiplicity~$m$ is given
by
$c_{b,l',m} := \sum_{s' \in \mathcal{S}'} \abs{d_{b, s'} - ml'_{s'}}$.
A lower bound for the cost $c_{b,l,m}$ of any completion $l$ of~$l'$
can be derived as follows: Denote by
\begin{equation}
  c_{b,l^*_s,m} := \min \bigl\{ \abs{d_{b, s} - ml_s} : \min_c \le l_s \le \max_c \bigr\}
\end{equation}
the minimal partial cost over all single-size partial lot-types $l_s$
for a given branch~$b$ and a given multiplicity~$m$.  Let $l^*_s$ be a
corresponding minimizer.  In other words, if branch~$b$ receives $m$
lots of some lot-type, then the lot-type component $l^*_s$ for
size~$s$ incurs the smallest possible cost contribution in size~$s$.
With this locally optimal fixing, we obtain a lower bound for any
completion $l$ of~$l'$ if the multiplicity is~$m$:
\begin{equation}
  c_{b,l,m} \ge \lambda_b(l', m) := c_{b,l',m} + \sum_{s \in \mathcal{S} \setminus
    \mathcal{S}'} c_{b,l^*_s,m}.
\end{equation}
A lower bound for the cost in branch $b$ of any completion~$l$ of~$l'$
with any multiplicity is then
\begin{equation}
  c_{b,l,m} \ge \lambda_b(l') := \min_{m \in \mathcal{M}} \lambda_b(l', m).
\end{equation}
Unfortunately, the function $\lambda_b(l', m)$ is not convex in~$m$.
Still, the optimization over $m \in \mathcal{M}$ can be done by
complete enumeration, since typically $\abs{M} \le 7$.

The depth-first-search branch-and-bound follows in each recursion
level the extensions of a partial lot-type by a single new size in the
order of increasing lower bounds for the cost of completion.  A
possible pseudo code is shown in Algorithms \ref{alg:FLB}
and~\ref{alg:recFLB}.  \deleted[id=JR, comment={the parameters are now
  explained elsewhere}]{For our tests, we have used $K = 3$.}

\subsection{The subroutine \protect\SFA{$\protect\ScoreTable$}}
\label{sec:subroutine-SFA}

We briefly recall the SFA heuristics from~\cite{p_median}\added[id=JR,
comment={indicated that we enhanced SFA}]{, adapted to our context}.
The name stems from the three main steps \emph{score -- fix --
  adjust}.  In the \emph{score} step, all lot-types are scored by
how-often they are the locally best, second-best, \replaced[id=JR,
comment={parametrized how many we use}]{\ldots, $K$th-best}{and
  third-best} lot-types for a branch. \added[id=JR, comment={moved the
  mention of the score table here}]{In our set-up, the scoring
  information is taken from the output $\protect\ScoreTable$ of
  $\protect\GLB$ in Algorithm~\ref{alg:GLB}}.  The \replaced[id=JR,
comment={concretized where the scoring comes from}]{order
  of~$\protect\ScoreTable$}{scoring} implies a lexicographic order on
all $k$-subsets of lot-types \added[id=JR,
comment={dto.}]{in~$\protect\ScoreTable$}.  In the \emph{fix} step,
all $k$-subsets of lot-types \replaced[id=JR, comment={explicit
  reference to the data structure}]{in~$\protect\ScoreTable$}{with a
  non-zero score} are traversed in this lexicographic order.  To each
branch, the best fitting lot-type from this $k$-subset is assigned in
the locally optimal multiplicity.  If the lower and upper bounds on
the total cardinality of the supply over all branches are satisfied,
this yields a feasible solution. If not, the \emph{adjust} step is
necessary: For each possible exchange of a multiplicity assignment
that makes the total-cardinality violation smaller (no overshooting
allowed), the relative cost-increase per cardinality change is
computed.  Then, the assignment from the \emph{fix} step is adjusted
by applying these exchange sequentially in the order of increasing
relative cost increases until either the total cardinality is feasible
or no feasible exchange is possible.  In the first case, a feasible
assignment is found and the cost is compared to the current best
incumbent.  In the second case, the $k$-subset under consideration is
dismissed.  This procedure is run until either all $k$-subsets of
\replaced[id=JR, comment={explicit reference to the data
  structure}]{lot-types in~$\protect\ScoreTable$}{non-zero-scored
  lot-types} have been processed or a time limit has been reached.
The best found solution is returned.  In this paper, we extend the
\emph{adjust} step by \deleted[id=JR, comment={simplified
  sentence}]{also allowing for} exchanges of lot-type-multiplicity
assignments.  This has the advantage that we can find feasible
solutions in more instances.  \deleted[id=JR,comment={the code uses
  only the list obtained by \protect\GLB}]{To obtain the scoring
  information in our context, we reuse \protect\ScoreTable{} of
  Algorithm~\ref{alg:GLB}.}  An overview in pseudo code can be found
in Algorithm~\ref{alg:SFA}.

\SetKwData{AdjTable}{AdjTable}
\SetKwData{Cost}{$z$}
\SetKwFunction{ALH}{AverageLottypeHeuristics}
\begin{algorithm}[htbp]
  \KwIn{A consistent SLDP, a sorted table $\ScoreTable$ of lot-types
    with non-zero score}
  \KwOut{A feasible solution and its objective value
    $\bigl(\Incumbent, \UpperBound\bigr)$ or $(-, \infty)$}

  $(\Incumbent, \UpperBound) \longleftarrow (-, \infty)$\;
  \uIf{$k > 1$}{
    $(\Incumbent, \UpperBound) \longleftarrow \ALH{}$\;
  }
  \Repeat{time limit reached}{
    \For{$L$ in $k$-subsets of lot-types in $\ScoreTable$}{
      \For{$b \in \mathcal{B}$}{
        $\bigl(b, l_L(b), m_L(b)\bigr) \longleftarrow (b, l^*, m^*)$
        \tcp*{locally optimal assignment}
      }
      \While{$\bigl(b, l_L(b), m_L(b)\bigr)$ violates cardinality constraint}{
        $\AdjTable \longleftarrow \bigl\{ \bigl(b, l, m\bigr)
        : \bigl(b, l, m\bigr) \text{ decreases violation of $\bigl(b, l_L(n), m_L(b)\bigr)$}
        \bigr\}$\;
        \uIf{$\AdjTable = \emptyset$}{
          $\Cost{L} \longleftarrow \infty$, 
          break\;
        }
        \Else{
          $(b^*, l^*, m^*) \longleftarrow$ the marginally cheapest
          element in $\AdjTable$\;
          $\bigl(b^*, l_L(b^*), m_L(b^*)\bigr) \longleftarrow (b^*, l^*, m^*)$\;
        }
      }
      $\Cost{L} \longleftarrow$ cost of assignment $\bigl(b, l_L(b),
      m_L(b)\bigr)_{b \in \mathcal{B}}$\;
      \If{$\Cost{L} < \UpperBound$}{
        $\UpperBound \longleftarrow \Cost{L}$,
        $\LiftedIncumbent \longleftarrow$ variable encoding of
        $\bigl(l_L(b), m_L(b)\bigr)_{b \in \mathcal{B}}$;
      }
    }
  }
  \Return{$\bigl(\LiftedIncumbent, \UpperBound\bigr)$}\;
  \caption{The subroutine \protect\SFA.}
  \label{alg:SFA}
\end{algorithm}

\SetKwData{AvgDemandType}{\ensuremath{\bar{d}_t}}
\SetKwData{AvgDemandComp}{\ensuremath{\bar{d}_c}}
\SetKwData{AggCard}{\ensuremath{\bar{n}}}
\SetKwData{AggDemand}{\ensuremath{D}}
\begin{algorithm}
  \KwIn{A consistent SLDP with $k > 1$}
  \KwOut{A feasible solution and its objective value
    $\bigl(\Incumbent, \UpperBound\bigr)$}
  $\AvgDemandType \longleftarrow
  \max\bigl\{
  \min_t,
  \min\bigl\{
  \max_t,
  \frac{1}{\abs{\mathcal{B}}}\sum_{b \in \mathcal{B}} \sum_{s
    \in \mathcal{S}} d_{b, s}
  \bigr\}
  \bigr\}$\;
  $\AvgDemandComp \longleftarrow
  \max\bigl\{
  \min_c,
  \min\bigl\{
  \max_c,
  \frac{1}{\abs{\mathcal{S}}}\AvgDemandType
  \bigr\}
  \bigr\}$\;
  $\check{l} \longleftarrow \bigl( \bigl\lfloor \AvgDemandComp \bigr\rfloor \bigr)_{s \in \mathcal{S}}$,
  $\hat{l} \longleftarrow \bigl(\bigl\lceil \AvgDemandComp \bigr\rceil \bigr)_{s \in \mathcal{S}}$\;
  \While{$\abs{\check{l}} < \min_t$}{
    $\check{l}_s \longleftarrow \check{l}_s + 1$ for some $s \in \mathcal{S}$ with $\check{l}_s < \max_c$\;
  }
  \While{$\abs{\hat{l}} > \max_t$}{
    $\hat{l}_s \longleftarrow \hat{l}_s - 1$ for some $s \in \mathcal{S}$ with $\hat{l}_s > \max_c$\;
  }
  $\AggCard \longleftarrow 0$, $\AggDemand \longleftarrow 0$, $m(b) \longleftarrow 1$\;
  \For{$b \in \mathcal{B}$}{
    \uIf{$\AggCard \abs{\check{l}} + (\abs{\mathcal{B}} - b - 1)\abs{\hat{l}} < \underline{I}$}{
      $l(b) \longleftarrow \hat{l}$\;
    }
    \uElseIf{$\AggCard \abs{\check{l}} + (\abs{\mathcal{B}} - b - 1)\abs{\check{l}} > \overline{I}$}{
      $l(b) \longleftarrow \check{l}$\;
    }
    \uElseIf{$\AggCard + \frac{\abs{\check{l}} + \abs{\hat{l}}}{2} > \AggDemand + d_b$}{
      $l(b) \longleftarrow \check{l}$\;
    }
    \Else{
      $l(b) \longleftarrow \hat{l}$\;
    }
    $\AggCard \longleftarrow \AggCard + \abs{l(b)}$,
    $\AggDemand \longleftarrow \AggDemand + d_b$\;
  }
  $\LiftedIncumbent \longleftarrow $ variable encoding of $\bigl(l(b),
  m(b) \bigr)_{b \in \mathcal{B}}$,
  $\UpperBound \longleftarrow $ cost of $\Incumbent$\;
  \Return{$\bigl(\LiftedIncumbent, \UpperBound\bigr)$}\;
  \caption{The subroutine \protect\ALH.}
  \label{alg:ALH}
\end{algorithm}

In order to find a feasible solution for all \replaced[id=JR,
comment={used the defined notion}]{consistent}{``reasonable''}
instances we \replaced[id=JR, comment={minor
  rephrasing}]{employ}{suggest} a simple fall-back heuristic called
\emph{average lot-type heuristics (ALH)} that guarantees feasibility
in all consistent instances with $k > 1$.  This is no serious
restriction because for $k = 1$ SFA was proven to be fast and optimal
in~\cite{p_median}.  We consider ALH as part of SFA that is run at the
very beginning of SFA.

ALH works as follows: Compute the average demand
$\bar{d}_t := \frac{1}{\abs{\mathcal{B}}} \sum_{b \in \mathcal{B}}
d_{b}$ per branch over all branches with
$d_b := \sum_{s \in \mathcal{S}} d_{b, s}$.  If this number is
smaller than~$\min_t$, set it to $\min_t$.  If it is larger
than~$\max_t$, set it to~$\max_t$.  Both cannot happen since then
$\min_t > \max_t$, and the instance is infeasible.  Next, compute
the average demand per size
$\bar{d}_c := \frac{1}{\abs{\mathcal{S}}} \bar{d}_t$.  If this
number is smaller than~$\min_c$ set it to~$\min_c$.  If it is larger
than~$\max_c$ set it to~$\max_c$.  Again, both cannot happen for a
feasible instance.

For each $s \in \mathcal{S}$ let
\begin{align}
  \check{l}_s &:= \max \bigl\{ \lfloor \bar{d}_c \rfloor, \min_c \bigr\}\text{ and }\\
  \hat{l}_s   &:= \min \bigl\{ \lceil \bar{d}_c \rceil, \max_c \bigr\}.
\end{align}
While $\abs{\check{l}} < \min_t$, increase an arbitrary component
$\check{l}_s < \max_c$ by one. (If no such component exists, the
instance is infeasible.)  Analogously, we adjust~$\hat{l}$.

This defines two applicable lot-types $\check{l}$ and~$\hat{l}$. If
$\abs{\mathcal{B}}\cdot \abs{\check{l}} > \overline{I}$, then, by
construction,
$\abs{\mathcal{B}}\cdot \abs{\mathcal{S}} \bar{d}_c > \overline{I}$, or
$\abs{\mathcal{B}} \cdot \abs{\mathcal{S}} \min_c > \overline{I}$, or
$\abs{\mathcal{B}} \bar{d}_t > \overline{I}$, or
$\abs{\mathcal{B}} \min_t > \overline{I}$.  All cases prove that the
instance is not consistent.  Similarly,
$\abs{\mathcal{B}} \cdot \abs{\hat{l}} < \underline{I}$ is impossible for
consistent instances.

Now assign $\check{l}$ and $\hat{l}$ sequentially to branches.  This
is possible for all proper instances ($k > 1$).  Assume, after the
assignment to $b-1$ branches we have assigned $n_{b-1}$ pieces and
the aggregated demand over the first $b-1$ branches is $D_{b-1}$.
If
$n_{b-1} + \abs{\check{l}} + (\abs{\mathcal{B}} - b -
1)\abs{\hat{l}} < \underline{I}$, we have to assign $\hat{l}$ to
branch~$b$.  If
$n_{b-1} + \abs{\hat{l}} + (\abs{\mathcal{B}} - b -
1)\abs{\check{l}} > \overline{I}$, we have to assign $\check{l}$ to
branch~$b$.  Otherwise, assign to branch~$b$ the lot-type
$\check{l}$ if and only if
$n_{b-1} + \frac{\abs{\check{l}} + \abs{\hat{l}}}{2} > D_{b-1} +
d_b$, i.e., try to approximate the demand aggregated up to
branch~$b$ as closely as possible by the number of assigned pieces
aggregated up to branch~$b$.

Whenever
$\overline{I} - \underline{I} \ge \abs{\hat{l}} - \abs{\check{l}}$
this sequential algorithm yields a feasible solution.  For example,
if $\overline{I} - \underline{I} \ge \max_t - \min_t$ this is
satisfied because, by construction,
$\max_t - \min_t \ge \abs{\hat{l}} - \abs{\check{l}}$.  The pseudo
code for ALH can be found in Algorithm~\ref{alg:ALH}.  In all of our
test instances, SFA found a feasible solution within the time limit
so that ALH was not really needed.

\subsection{The subroutine \protect\RMPinit{$K, K', K''$}}
\label{sec:subrouting-RMPinit}

In this subroutine, we generate an initial set of columns for
the RMP.  Our choice is to add
\begin{enumerate}
\item the $x$-columns corresponding to the SFA solution and the
  implied $y$-columns for all lot-types used in the SFA solution;
\item the $y$-columns corresponding to the $K$
  highest-scored lot-types, and for each of these $K$ lot-types the
  $x$-columns for all branches with the locally optimal
  multiplicities;
\item for each branch $b$, the $x$-columns corresponding to all
  lot-types $l$ that are among the locally best $K'$
  lot-types for~$b$ and the corresponding locally optimal
  lot-type-multiplicities, together with the implied $y$-columns.
\item \added[id=JR, comment={added this item based on the actual
    software}]{If for the same lot-type and branch, there are multiple
    multiplicities in the initial RMP, we add the intermediate
    multiplicities as well; the resulting interval of multiplicities
    is then enlarged by $K''$ multiplicities to the left and the right
    up to the boundaries of~$\mathcal{M}$.}
\end{enumerate}
The reason for the first set of columns is that we want the RSLDP to
find a solution at least as good as the heuristic.  The reason for the
second set of columns is to consider some lot-types for \emph{all}
branches, namely those that fit well a large number of branches; this
should ease the satisfaction of the cardinality constraint on the
number of used lot-types.  The reason for the third set of columns is
that the individually best lot-types for branches are promising for a
``spot-repair'' of almost complete assignments.  \added[id=JR]{The
  reason for adding the intermediate multiplicities is that this makes
  some (though not all) of the subproblems convex.} The computation
times in our test instances (see
section~\ref{sec_results_column_generation}) did not react too
sensitive on the deliberate choices we made here.  We tested smaller
and larger sets of columns with no consistent improvement in our
tests.

\subsection{The subroutine \protect\RMP{$\protect\RMPTotalCols; \mathcal{L}''$}}
\label{sec:subrouting-RMPsolve}

In this subroutine, we solve the current RMP containing $y$-variables
corresponding \added[id=SK]{to} a subset of lot-types~$\mathcal{L}'$
and $x$-variables for \added[id=SK]{assignment} options corresponding
to $\RMPTotalCols$.  The subset of lot-types $\mathcal{L}''$ was
considered in the previous RSLDP, i.e., it is missing in the support
of the set-covering constraint of the current RMP.

The only important choice to make in this procedure is which LP
algorithm we call in a black-box LP solver.  We chose to employ, as
usual,
\begin{enumerate}
\item the primal simplex algorithm in the pricing phase (because
  adding columns maintains the primal feasibility of the previous
  optimal basis);
\item the dual simplex algorithm in the cutting phase (because the
  strengthening of the set-covering constraint maintains the
  dual feasibility of the previous optimal basis).
\end{enumerate}
The outcome of this subroutine is a set of dual variables together
with an objective value
$(\alpha, \beta, \gamma, \delta, \mu, \phi, \psi; z^{\mathit{RMP}})$
that is optimal for the current DRMP.  Note, that the objective
function value $z^{\mathit{RMP}}$ of the RMP and the DRMP is, in
general, neither an upper nor a lower bound for the optimal SLDP
value.

\subsection{The subroutine \protect\PP{$\alpha, \beta, \gamma,
    \delta, \mu, \phi, \psi; z^{\mathit{RMP}}; \mathcal{L}';
    \protect\LowerBound; K, K'$}}
\label{sec:subrouting-PP}

In this subroutine, the pricing problem (PP) is solved based on the
optimal dual variables and the objective value from the previous
solving of the RMP.  We follow the three types of sets of promising
variables from section~\ref{sec:theory}.

\SetKwFunction{CharShiftBound}{$z^{\textit{CSB}}$}
\SetKwFunction{FPL}{FindPromLottypes}
\begin{algorithm}[htbp]
  \KwIn{Duals/objective $\alpha, \beta, \gamma, \delta, \mu, \phi, \psi;
    z^{\mathit{RMP}}$ optimal for the DRMP with lot-types
    $\mathcal{L}'$, current lower bound $\LowerBound$;
    $K, K' \in \mathbb{N}$} 
  \KwOut{New variables and a new lower bound $\bigl(\RMPNewCols,
    \LowerBound\bigr)$} 

  $\RMPNewCols \longleftarrow \emptyset$\;
  \For{$b \in \mathcal{B}$}{
    \For{$l \in \lottypesatbranch$}{
      \If{$\min_{m \in \mathcal{M}}\bar{c}_{b, l, m} < 0$}{
        \tcp{let $m^*$ be a minimizer:}
        $\RMPNewCols \longleftarrow (b, l, m^*)$\;
      }
    }
  }
  \If{$\RMPNewCols \neq \emptyset$}{
    \Return{$\bigl(\RMPNewCols, \LowerBound\bigr)$}\;
  }
  \For{$l \in \mathcal{L}' \setminus \bigcap_{b \in \mathcal{B}}
    \lottypesatbranch$}{
    \If{$\check{c}_{l} 
      < \sum_{\substack{b \in \mathcal{B}:\\l \in \lottypesatbranch}}
      \beta_{b,l} - \gamma - \delta_l + \mu$}{
      \tcp{let $m^*$ be a minimizer for $b$ in $\min_{m \in
          \mathcal{M}}\bar{c}_{b, l, m}$:}
      \If{$\abs{\mathcal{L}^{new}} < K$ or $\check{c}_{l} < \check{c}_{l_K}$}{
        $\RMPNewCols \longleftarrow \RMPNewCols \cup
        \Bigl\{\bigl(b, l, m^*\bigr) : b \in \mathcal{B},
        l \in \mathcal{L}' \setminus \lottypesatbranch\Bigr\}$\;
      }
    }
  }
  \If{$\RMPNewCols \neq \emptyset$}{
    \Return{$\bigl(\RMPNewCols, \LowerBound\bigr)$}\;
  }
  $\RMPNewCols \longleftarrow \RMPNewCols \cup \FPL{$\alpha, \beta,
    \gamma, \delta, \mu, \phi, \psi; z^{\textit{RMP}}; \mathcal{L}'; K'$}$\;
  \If{$\CharShiftBound > \LowerBound$}{
    $\LowerBound \longleftarrow \CharShiftBound$\;
  }
  \Return{$\bigl(\RMPNewCols, \LowerBound\bigr)$}\;
  \caption{\protect\PP}\label{alg:PP}
\end{algorithm}

\SetKwFunction{RecFindProm}{RecPromLottypes}
\begin{algorithm}[htbp]
  \KwIn{Duals/objective $\alpha, \beta, \gamma, \delta, \mu, \phi, \psi;
    z^{\mathit{RMP}}$ optimal for the DRMP with lot-types
    $\mathcal{L}'$, $K \in \mathbb{N}$}
  \KwOut{New columns $\RMPNewCols$ for the $K$ most promising new lot-types}
  \tcp{start with an empty sorted list:}
  $\RMPNewCols \longleftarrow \bigl(\bigr)$\;
  \tcp{start with an empty partial lot-type:}
  $l' \longleftarrow \bigl(\bigr)$\;
  \tcp{applicable range of cardinalities for the first size:}
  $n_{\min} \longleftarrow \min_t - (\abs{S} - 1) \max_c$\;
  $n_{\max} \longleftarrow \max_t - (\abs{S} - 1) \min_c$\;
  $\RecFindProm{$\alpha, \beta, \gamma, \delta, \mu, \phi, \psi;
    z^{\mathit{RMP}}; \mathcal{L}'; K; l', n_{\min}, n_{\max}; \RMPNewCols$}$\;
  \Return{$\RMPNewCols$}\;
  \caption{\protect\FPL}\label{alg:FPL}
\end{algorithm}

\SetKwFunction{BBRedLowerBound}{$\bar{\lambda}(l'')$}
\begin{algorithm}[htbp]
  \KwIn{$\alpha, \beta, \gamma, \delta, \mu, \phi, \psi;
    z^{\mathit{RMP}}$ optimal for the DMP with lot-types
    $\mathcal{L}'$, $K \in \mathbb{N}$, a 
    partial lot-type~$l'$, $n_{\min}, n_{\max} \in \mathbb{N}$, a
    set $\RMPNewCols$ of columns corresponding to lot-types
    found so far with $\mathcal{L}' = (l_1, \dots,
    l_K)$ sorted by sum of negative reduced costs}
  \KwOut{An updated set $\RMPNewCols$ of columns}
  
  \uIf{$l'$ is complete}{
    \tcp{lot-type reached - check cost and update lot-type list:}
    \If{$l' \in \mathcal{L} \setminus \mathcal{L}'$}{
      \If{$\abs{\mathcal{L}^{\mathit{new}}} < K$ or $\check{c}_{l'} <
        \check{c}_{l_K}$}{
        $\RMPNewCols \longleftarrow \RMPNewCols \cup
        \bigl\{
        (b, l', m^*) : b \in \mathcal{B}: \bar{c}_{b, l, m^*} < 0
        \bigr\}$\;
      }
    }
  }
  \Else{ 
    \tcp{list extensions of $l'$ sorted by lower bounds:}
    $\CList \longleftarrow \bigl(\bigr)$\;
    \For{$n$ from $n_{\min}$ to $n_{\max}$}{
      $l'' \longleftarrow \bigl(l', n \bigr)$\;
      \If{$\BBRedLowerBound < -\gamma + \mu$ and $\BBRedLowerBound
        < \check{c}_{l_K}$}{
        \tcp{lower bound does not exclude $l''$:}
        insert $n$ into $\CList$\;
      }
    }
    sort $\CList$ by increasing $\BBRedLowerBound$\;
    \For{$n$ in $\CList$}{
      \tcp{process all non-pruned extensions:}
      $l'' \longleftarrow \bigl(l', n \bigr)$\;
      $n'_{\min} \longleftarrow n_{\min} + (\max_c - n)$\;
      $n'_{\max} \longleftarrow n_{\max} - (n - \min_c)$\;
      $\RecFindProm{$\alpha, \beta, \gamma, \delta, \mu, \phi, \psi;
        z^{\mathit{RMP}}; \mathcal{L}'; K; l', n_{\min}, n_{\max}; \RMPNewCols$}$\;
    }
  }
  \caption{\protect\RecFindProm}
  \label{alg:recPP}
\end{algorithm}

First, we check whether there are promising $x_{b, l, m}$ with
$\bar{c}_{b, l, m} < 0$ for which $l \in \lottypesatbranch$ and
$m \in \mathcal{M} \setminus \multsatbranchlottype$.  We add all
such variables to the RMP and store for each branch the minimal
reduced cost observed.

Next, we check for all
$l \in \mathcal{L}' \setminus \bigcap_{b \in \mathcal{B}}
\lottypesatbranch$ whether
\begin{equation}
  \check{c}_l = -\sum_{\substack{b \in \mathcal{B}:\\l \in \mathcal{L}' \setminus \lottypesatbranch}}
  \bigl( \min_{m \in \mathcal{M}} \bar{c}_{b,l,m} \bigr)^- < \sum_{\substack{b
      \in \mathcal{B}:\\l \in \lottypesatbranch}}\beta_{b,l} - \gamma - \delta_l + \mu.
\end{equation}
For any $l$ for which this happens we add all $x_{b, l, m}$ with
$\min_{m \in \mathcal{M}} \bar{c}_{b,l,m} < 0$ to the RMP.
\added[id=JR, comment={made this compatible with the actual
  software}]{We support a parameter $K \in \mathbb{N}$ to restrict the
  output to the $K$ column sets that have the smallest~$\check{c}_l$.}

For the final case, if there are many lot-types, we cannot simply
check for all $l \in \mathcal{L} \setminus \mathcal{L}'$ whether
\begin{equation}
  \label{eq:newlot-type}
  \check{c}_l   = -\sum_{b \in \mathcal{B}} \bigl( \min_{m \in \mathcal{M}} \bar{c}_{b,l,m}
  \bigr)^- + \bigl( \min_{\substack{b \in \mathcal{B}\\m \in \mathcal{M}}}
  \bar{c}_{b,l,m} \bigr)^+ < -\gamma + \mu.
\end{equation}
Therefore, for some $K' \in \mathbb{N}$ we solve the
optimization problem to find the (at most) $K'$ lot-types
$l \in \mathcal{L} \setminus \mathcal{L}'$ with minimal $\check{c}_l$
among those satisfying \eqref{eq:newlot-type}.  Then, we add all
corresponding $y_l$ and all $x_{b, l, m}$ with
$\min_{m \in \mathcal{M}} \bar{c}_{b,l,m} < 0$ to the RMP.

To solve the optimization problem, we devise a branch-and-bound
algorithm extending partial lot-types that is very similar to the
generation of the $K$ locally best lot-types in
Algorithm~\ref{alg:GLB}.  The only difference is that instead of
optimizing the cost for a single branch, we optimize the sum of
negative reduced costs over all branches.  This requires a little care
for the lower bound.

For a given branch~$b$ and a partial lot-type $l'$ with sizes in
$\mathcal{S}'$, we use the lower bound $\lambda_b(l', m)$ from
subsection~\ref{alg:GLB} for the primal cost of any completion of
$l'$ to a complete lot-type when delivered to branch $b$ in
multiplicity~$m$.  In order to turn this into a bound for the
reduced cost of any completion, we subtract the duals from the
canonical lifting.  This yields:
\begin{align}
  \bar{c}_{b, l', m}
  &\ge
    \bar{\lambda}_b(l', m)\notag\\
  &:= \lambda_b(l', m) - \alpha_b
    - m \max_{l_{\mathcal{S}'} = l'} \abs{l}(\phi - \psi)^+
    + m \min_{l_{\mathcal{S}'} = l'} \abs{l}(\phi - \psi)^-.
\end{align}
Together with
\begin{align}
  \max_{l_{\mathcal{S}'} = l'} \abs{l}
  &=
    \abs{l'} +
    \min \Bigl\{
    \bigl( \abs{\mathcal{S}} - \abs{\mathcal{S}'} \bigr) \max_c,
    \max_t - \abs{l'}
    \Bigr\},\\
  \min_{l_{\mathcal{S}'} = l'} \abs{l}
  &=
    \abs{l'} +
    \max \Bigl\{
    \bigl( \abs{\mathcal{S}} - \abs{\mathcal{S}'} \bigr) \min_c,
    \min_t - \abs{l'}
    \Bigr\},
\end{align}
we receive a lower bound for the reduced cost for a given branch and
a given multiplicity for any completion $l$ of~$l'$.  Now we
minimize over all multiplicities for each branch separately and take the sum. %% sum up.  
Thus, for each completion $l$ of~$l'$ we have, using the
short-hand $m^* := m^*(b, l)$ for minimizing multiplicities:
\begin{align}
  \check{c}_l
  &=
    \sum_{b \in \mathcal{B}}
    -\bigl( \bar{c}_{b, l, m^*} \bigr)^-
    +
    \bigl( \min_{b \in \mathcal{B}} \bar{c}_{b, l, m^*} \bigr)^+\\
  &\ge
    \sum_{b \in \mathcal{B}} -\bigl( \bar{\lambda}_b\bigl(l',  m^*\bigl) \bigr)^-
    +
    \min_{b \in \mathcal{B}}\bigl( \bar{\lambda}_b\bigl(l', m^*\bigl) \bigr)^+\\
  &:=
    \bar{\lambda}(l').
\end{align}
Again, the function $\bar{\lambda}_b(l', m)$ is not convex in~$m$,
and we implement the optimization over $m \in \mathcal{M}$ by
complete enumeration (recall that $\abs{M} \le 7$ typically).

From all the newly generated columns, we compute the characteristic
shift bound~$z^{\textit{CSB}}$ using Theorem~\ref{thm:bound} and
update the lower bound $\LowerBound$ if
$z^{\textit{CSB}} > \LowerBound$, i.e., the new bound is tighter.
Algorithms \ref{alg:PP} through~\ref{alg:recPP} list a possible pseudo
code for this.  \deleted[id=JR, comment={parameter settings are
  discussed elsewhere}]{For our tests, we have used $K = 3$.}

In our test, we used \emph{incomplete pricing}, i.e., if we have found
promising sets of variables with $l \in \lottypesatbranch$, then we
stop; if not, we look for promising sets of variables with
$l \in \mathcal{L}' \setminus \bigcup_{b \in
  \mathcal{B}}\lottypesatbranch$; if we find promising sets of
variables there, then again we stop.  If then we still have not found
a promising set of variables, then we enter the more time-consuming
search for new lot-types $l \in \mathcal{L} \setminus \mathcal{L}$.
Note, that the characteristic shift bound can only be derived from
complete pricing steps.

\subsection{The subroutine \protect\RSLDP{$\mathcal{L}'$}}
\label{alg:rsldp}

This subroutine builds an RSLDP with all variables corresponding to
all branches $b \in \mathcal{B}$, lot-types $l \in \mathcal{L}'$, and
multiplicities $m \in \mathcal{M}$.  This RSLDP is then solved by the
black-box ILP solver.  There is one specialty for our implementation:
We used the original model from Section~\ref{sec:modelling} without
the augmenting constraints.  This means, the RSLDP will reproduce
earlier solutions, which causes no harm.  The reason for this is simply
that our particular solver (\texttt{cplex}) could prove the
optimality of an old solution in the smaller model more quickly than the
infeasibility of the larger model with a cut-off at the current upper
bound~$\UpperBound$.  Since this can be dependent on the solver
software and even the version of it, we cannot make a general
statement about which option is to be preferred. %%what to prefer.  
\deleted[id=JR, comment={removed the
  remark that noone understood}]{It is just a guess, but it seems that
  current solver technology can easier prove optimality than
  infeasibility.}

\section{Computational results}
\label{sec_results_column_generation}

\noindent
In this section we compare the algorithm ASG, implemented in C++, with
the static solution of the ILP model from
Section~\ref{sec:modelling}\added[id=JR, comment={introduced a name to
  simplify further reference to this algorithm}]{, denoted by
  \STATICCPLEX}.

The computational environment was as follows.  The computer was an
iMac (Retina 5K, 27inches, 2019) 3 GHz Intel Core i5 with 32GB 2667
MHz DDR4 RAM running \texttt{MacOSX 10.15.5} based on \texttt{Darwin
  Kernel Version 19.5.0}.  We used \texttt{gcc/g++} from \texttt{Apple
  clang version 11.0.3 (clang-1103.0.32.62)}. As an LP- and ILP-solver
we used \texttt{IBM ILOG CPLEX 12.9.0.0}, both for \STATICCPLEX{} and
for the RMP and RSLDP during the run of ASG.  Profiling was performed
using the \texttt{clock()} utility from \texttt{time.h}.  Timing was
consistent to an accuracy of $\sim \pm 5\%$.  We performed all test
for zero with $\epsilon = 10^{-9}$.  \deleted[id=JR,comment={High
  solver accuracy to large can lead to tailing-off insider the RSLDP;
  high accuracy is irrelevant anyway.}]{The accuracy of the solver was
  set to the same value.  Because of scaling effects, we can expect
  the results to be accurate up to a relative error of $10^{-8}$.}  We
used the defaults for all solver parameters.  \added[id=JR,
comment={corrected the statement about the accuracy}]{Thus, the
  solutions can not be expected to be more accurate than the solver's
  largest default accuracy ``\texttt{mipgap}'' of $10^{-4}$, which is
  still far more accurate than the demand data.}

For ASG, we used a time limit of
$\max\{1.0, \log(\abs{\mathcal{L}}) / 10.0\}$ for SFA.  \added[id=JR,
comment={discussed the investigated parameter settings}]{Moreover, in
  the ``default'' parameter setting of ASG we use $K_1 = 3$ (for the
  maximal number of locally best lot-types determined for each
  branch), $K_2 = 3$ (for the maximal number of used lot-types from
  the score table), $K_3 = 0$ (for the maximal number of additional
  multiplicities), $K_4 = 10$ (for the maximal number of promising
  sets of columns for new branches assigned to an already generated
  lot-type), and $K_5 = 3$ (for the number of promising sets of
  columns with not yet generated lot-types).  In a comparison, we
  checked the influence of parameter settings representing ``fewer
  columns'' ($K_1 = 1$, $K_2 = 3$, $K_3 = 0$, $K_4 = 1$, and
  $K_5 = 1$) and ``more columns'' ($K_1 = 10$, $K_2 = 5$, $K_3 = 2$,
  $K_4 = \infty$, and $K_5 = 10$).}  Furthermore, we used incomplete
pricing to avoid the time-consuming branch-and-bound search for
promising new lot-types whenever \replaced[id=JR, comment={more
  specific phrasing}]{other promising columns could be
  generated}{possible}.

Our tests encompass two substantially different sets of instances of
the \replaced[id=JR, comment={the experiments now contain solutions to
  the SLDP as well}]{(S)LDP}{LDP}.  \deleted[id=JR,
comment={dto.}]{Since the SLDP can be reduced to an LDP (see
  Section~\ref{sec_formal_problem_statement}), we used only LDP
  instances for our test to reduce the influence of too many
  parameters. The first set of LDPs contains 860 real-world instances
  from November 2011 from the daily production of our industry partner
  from different apparel groups.  The second set of LDPs contains $5$
  groups of $9$ random instances with identical parameters each.}
\added[id=JR,comment={Since one reviewer did not understand the
  difference between the nominal LDP and the equivalent LDP of an
  SLDP, we show here the difference.}]{First, we used a set of $860$
  real-world production instances with various instance parameters.
  All instances have three scenarios for the demands.  One of these
  scenarios represents the nominal scenario, as defined in
  Section~\ref{sec:stoch-lot-type}.  For all real-world instances, we
  solved the SLDP (in terms of the equivalent LDP) and the nominal
  LDP, both by ASG and \STATICCPLEX.  To illustrate the gap between
  the SLDP and the nominal LDP we report on the difference in the
  objective function values and the value of the stochastic solution
  (VSS)~\cite[Section~4.2]{1223.90001}.  Moreover, we compared the
  run-times and the numbers of used columns for the two competing
  algorithms.  Second, we compare the algorithms ASG and
  \STATICCPLEX{} on a set of $5$ groups of $9$ random LDP-instances
  with identical parameters each.  The motivation for this second set
  of instances is to safe-guard against the possibility that the
  results on real-world data are an artifact of some hidden structure
  in our special instances.  Since the real-world instances exhibit
  only minor differences in cpu times between solving the nominal LDPs
  and the SLDPs, we restricted the tests for the random instances to
  the nominal LDP.}

In the $860$ real-world instances, the numbers of applicable lot-types
range from $9$ \replaced[id=JR, comment={better
  line-break}]{through}{to} $1{,}198{,}774{,}720$, the numbers of
branches from $1370$ to $1686${,} and the numbers of sizes from $2$
to~$7$.  In the static model, the numbers of variables range from
$8381$ to $9{,}867{,}114{,}720{,}320$, and the numbers of constraints
from $3355$ to $1{,}973{,}183{,}190{,}769$.  About half of the
instances have at least $125{,}038{,}518$ variables and
$25{,}006{,}257$ constraints.  \replaced[id=JR, comment={more precise
  reasoning}]{Since the larger instances are too large to read the
  complete static ILP-model into our computer}{\added[id=SK]{For
    computational reasons}}, we \replaced[id=JR, comment={used the
  name \STATICCPLEX}]{used \STATICCPLEX{}}{solved the static model}
only for those instances with at most $2000$ applicable lot-types,
leading to $334$ instances with $8381$ to $15{,}344{,}420$ variables
and $3355$ to $3{,}070{,}209$ constraints.  In this set, there were
also $8$ obviously infeasible instances for various reasons that can
be attributed to errors in the data handling or pathologic historic
demand data at the time.  We implemented a straight-forward procedure
to filter them out automatically.  ASG's memory consumption stayed
well below 1GB RAM, thanks to an implicit encoding of lot-types via
their lexicographic index in our implementation.  \added[id=JR,
comment={explained the variation of problem data parameters in the
  production data}]{Note that the wide range of parameters in the
  daily production data stems from the fact that different ware groups
  have different numbers of sizes and different restrictions for the
  numbers of pieces in the lot-types. For example, while winter coats
  usually come in four sizes with only very few pieces allowed in a
  lot-type,\deleted[id=JR, comment={grammar correction}]{ whereas} womens' underwear comes in twelve
  sizes with many pieces possible in a lot-type.  Some of the
  instances are similar in the optimal objective function value
  because the demand data has been estimated aggregated over the whole
  ware group (the sample sizes for estimation are too small
  otherwise).}

In the $5 \times 9$ random instances, the parameters in the $5$ groups are
given in Table~\ref{tab:randompars}.  The parameter sets are similar to those
found in the real data.  For these $5$ sets of parameters, we generated $9$
sets of demand data uniformly at random so that the resulting instances were
consistent.
\begin{table}[htbp]
  \begin{center}
    \footnotesize\sffamily
    \begin{tabular}{*{7}{r}}
      \toprule
      $\abs{\mathcal{B}}$ &  $\abs{\mathcal{S}}$ &  $\abs{\mathcal{M}}$ & $\min_c$ &  $\max_c$ &  $\min_t$ & $\max_t$\\
      \midrule 
      10 & 4 & 5 & 0 & 2 & 4   & 8   \\
      10 & 4 & 5 & 0 & 5 & 3   & 15  \\
      1303 & 4 & 5 & 0 & 5 & 3 & 15  \\      
      1328 & 7 & 5 & 1 & 3 & 7 & 14  \\    
      682 & 12 & 5 & 0 & 5 & 12& 30  \\
      \bottomrule
    \end{tabular}
    
    \bigskip
    \begin{tabular}{*{4}{r}}
      \toprule
      $\abs{\mathcal{L}}$ &  $k$ & \#vars &  \#cons\\
      \midrule
      50 & 3 & 2550 & 513\\
      1211 & 5 & 61{,}761 & 12{,}123\\
      1211 & 5 & 7{,}890{,}876 & 1{,}579{,}239\\
      1290 & 4 & 8{,}566{,}890 & 1{,}714{,}451\\
      1{,}159{,}533{,}584 & 5 & 3{,}955{,}169{,}055{,}024 & 790{,}801{,}904{,}973\\
      \bottomrule
    \end{tabular}
  \end{center}
  \caption{The parameters of the $5 \times 9 = 45$
    \added[id=JR]{random} instances.}
  \label{tab:randompars}
\end{table}

Let us first discuss the results on the $860$ real-world instances.  In those
instances that were solved both statically and dynamically we obtained
identical optimal values up to the expected numerical accuracy.

\begin{table}[H]
  \begin{center}
    \footnotesize\sffamily
    \begin{tabular}{lrrr}
      \toprule
      metric & average & median & maximum \\
      \midrule
      ELDP – LDP & 1638.48\lefteqn{\,} & 1654.25\lefteqn{\,} & 5033.04\lefteqn{\,}\\
      (ELDP – LDP) / ELDP & 66.69\lefteqn{\,\%} & 66.40\lefteqn{\,\%} & 200.59\lefteqn{\,\%}\\
      VSS = ELDP – SLDP & 15.69\lefteqn{\,} & 6.49\lefteqn{\,} & 263.34\lefteqn{\,}\\
      relative VSS = (ELDP – SLDP) / ELDP & 0.36\lefteqn{\,\%} & 0.15\lefteqn{\,\%} & 4.51\lefteqn{\,\%}\\
      \bottomrule
    \end{tabular}
  \end{center}
  \caption[SLDP versus nominal LDP]{Various metrics indicating the
    differences in the model optimal solutions for the SLDP and the
    nominal LDP: LDP denotes the optimal objective of the nominal LDP;
    ELDP denotes the expected cost of the nominal LDP solution over
    all three scenarios; SLDP denotes the optimal objective of the
    SLDP; VSS is the value of the stochastic solution}
  \label{tab:SLDPvsLDP}
\end{table}

\added[id=JR, comment={rewrote the observations}]{We compared ASG and
  \STATICCPLEX{} for the SLDP (via the equivalent LDP) and the nominal
  LDP in Table~\ref{tab:SLDPvsLDP}.  It can be seen that solving the
  nominal LDP instead of the SLDP significantly underestimates the
  expected cost.  However, the quality of its optimal solution
  evaluated in the SLDP is almost optimal: the relative value of the
  stochastic solution (relative VSS) is 0.36\,\% on average; no
  instance showed a relative VSS of more than 4.51\,\%. That means, if
  underestimation of cost and the rare occurrence of a larger relative
  VSS are not an issue, one can safely solve the nominal LDP instead
  of the SLDP.  It is a different question whether this pays off, as
  the next results show.}

\begin{table}[H]
  \begin{center}
    \footnotesize\sffamily
    \begin{tabular}{l@{\quad\quad}rrr@{\quad\quad}rrr}
      \toprule
      metric & \multicolumn{3}{c}{ASG on nominal LDP} & \multicolumn{3}{c}{ASG on SLDP}\\
             & default & fewer col's & more col's & default & fewer col's & more col's \\
      \midrule
      median cpu-time (small) & 4.00\lefteqn{\,s} & 8.08\lefteqn{\,s} & 9.68\lefteqn{\,s} & 10.90\lefteqn{\,s} & 10.79\lefteqn{\,s} & 11.54\lefteqn{\,s}\\
      percentage of~\STATICCPLEX & 0.67\lefteqn{\,\%} & 0.81\lefteqn{\,\%} & 0.81\lefteqn{\,\%} & 1.13\lefteqn{\,\%} & 0.96\lefteqn{\,\%} & 1.02\lefteqn{\,\%}\\
      median cpu-time (all) & 12.21\lefteqn{\,s} & 12.68\lefteqn{\,s} & 14.23\lefteqn{\,s} & 13.39\lefteqn{\,s} & 13.30\lefteqn{\,s} & 15.28\lefteqn{\,s}\\
      median cpu-time (large) & 22.84\lefteqn{\,s} & 25.60\lefteqn{\,s} & 33.13\lefteqn{\,s} & 26.13\lefteqn{\,s} & 32.52\lefteqn{\,s} & 29.83\lefteqn{\,s}\\
      \bottomrule
    \end{tabular}
  \end{center}
  \caption[Cpu-time comparison of ASG parameters]{Cpu-time comparison for
    various parameter settings for ASG solving the SLDP or the nominal
    LDP; here, the ``small'' instances are the ones that were solved
    by both \STATICCPLEX{} and ASG (at most 1820 applicable
    lot-types), the ``large'' instances are the 100 largest ones
    (134,596 through 1,198,774,720 applicable lot-types)}
  \label{tab:ASGtimes}
\end{table}

\added[id=JR, comment={rewrote the observations}]{Table~\ref{tab:ASGtimes} compares cpu-times on three
  parameter settings for ASG, both for solving the nominal LDP and the
  SLDP.  Over all instances in the default setting, solving the
  nominal LDP was around 10\,\% faster on average than solving the
  SLDP.  Still, all cpu-times are usable in practice.  It is a matter
  of the user's preferences whether to pay the 10\,\% extra cpu-time
  to obtain the better cost predictions of the SLDP.  Second, there is
  no unique fastest parameter setting.  Since in no category (LDP or
  SLDP; small, all, or large instances) ``more columns'' is the best,
  it is safe to go with ``default'' or ``fewer columns'', where the
  ``default setting'' has the edge for the large instances.  For us
  the most important result is that for \emph{all three} investigated
  and substantially different parameters settings ASG outperforms
  \STATICCPLEX{} by a median-factor of around 100 and solves even the
  large instances in only mildly longer cpu-times.  In other words,
  our main result does not at all depend on the parameter setting.}

\begin{table}[H]
  \begin{center}
    \footnotesize\sffamily
    \begin{tabular}{l@{\quad\quad}rrr@{\quad\quad}rrr}
      \toprule
      metric & \multicolumn{3}{c}{ASG on nominal LDP} & \multicolumn{3}{c}{ASG on SLDP}\\
             & default & fewer col's & more col's & default & fewer col's & more col's \\
      \midrule
      median no.~of cols (small) & 6461 & 6404 & 28553 & 6846 & 6634 & 28676\\    
      percentage of~\STATICCPLEX &    0.18\lefteqn{\,\%} & 0.18\lefteqn{\,\%} & 0.78\lefteqn{\,\%} & 0.19\lefteqn{\,\%} & 0.18\lefteqn{\,\%} & 0.79\lefteqn{\,\%}\\
      median no.~of cols (all) &   7050 & 6916 & 29089 & 7280 & 7072 & 29338\\    
      median no.~of cols (large) & 7136 & 6871 & 29362 & 7234 & 6981 & 29476\\
      \bottomrule
    \end{tabular}
  \end{center}
  \caption[Columns comparison of ASG parameters]{Comparison of the numbers of
    columns for various parameter settings for ASG solving the SLDP or
    the nominal LDP; here, the ``small'' instances are the ones that
    were solved by both \STATICCPLEX{} and ASG (at most 1820
    applicable lot-types), the ``large'' instances are the 100 largest
    ones (134,596 through 1,198,774,720 applicable lot-types)}
  \label{tab:ASGcols}
\end{table}

\added[id=JR, comment={rewrote the
  observations}]{Table~\ref{tab:ASGcols} compares the number of
  generated columns until optimality is proved by ASG. On average over
  all instances in the ``default'' setting, the nominal LDP needs
  3\,\% fewer columns than the SLDP.  The parameter setting of ``more
  columns'' leads to an unnecessary large number of columns, whereas
  ``default'' and ``few columns'' generate similar amounts of columns,
  making both of them perfectly usable in practice.}

\added[id=JR, comment={rewrote the observations}]{Summary: The
  analysis of the different parameter settings shows that ASG works
  quickly %%fast 
  for all three settings, though ``more columns'' wastes some
  memory for unnecessary columns.  Moreover, solving the SLDP has the
  advantage over the nominal LDP of a precise cost prediction at a
  small cpu-time overhead, but the quality of the SLDP solution is
  only insignificantly better than the nominal LDP solution.}

\added[id=JR, comment={rewrote the observations}]{After this summary
  of results aggregated over all 860 real-world instances for both the
  nominal LDP and the SLDP, we have a look at the performances of ASG
  and \STATICCPLEX{} depending on the number of applicable lot-types,
  which is, in our opinion, the single-most important parameter for
  the size of an instance.  We restrict ourselves to the solution of
  the nominal LDP in the following.}

\begin{figure}[H]
  \centering
  \includegraphics[width=\linewidth]{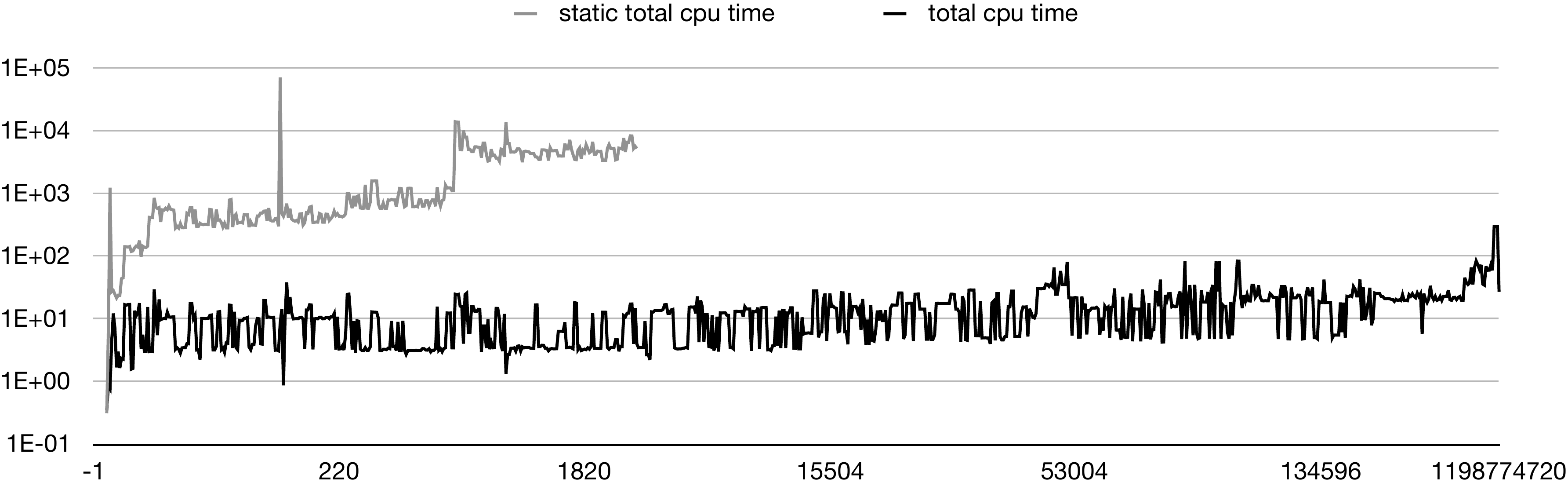}
  \caption{Comparison of cpu times in seconds on 860 real-world
    instances ordered by the numbers of applicable lot-types
    (indicated \replaced[id=JR]{by the grey line}{on the category
      axis})\added[id=JR]{; note that we connected all data points by
      straight lines for better visibility, although we have a
      discrete number of $860$ instances on the category axis}}
  \label{fig:results-real-world-cpu}
\end{figure}
Figure~\ref{fig:results-real-world-cpu} shows the cpu times for
\replaced[id=JR, comment={used \STATICCPLEX}]{\STATICCPLEX}{the static
  solution} and the ASG solution.  The instances have been
\added[id=JR, comment={minor rephrasing}]{weakly} ordered from left to
right by their numbers of applicable lot-types.  \replaced[id=JR,
comment={visualized the no.~of applicable lot-types by its own
  graph}]{The numbers of applicable lot-types is shown by the grey
  graph in the chart.}{Some of these numbers have been indicated on
  the category axis.}  The scale for the time axis in seconds is
logarithmic.  We see that ASG is consistently two to three orders of
magnitude faster than \replaced[id=JR, comment={used
  \STATICCPLEX}]{\STATICCPLEX}{static solving}. It is striking that
the cpu time of ASG is essentially stable around $10$ seconds, usually
never exceeding $100$ seconds over the whole range of instances with
only very few instances over $100$ seconds.  \added[id=JR,
comment={rewrote the observations}]{In contrast to this, the cpu times
  of \STATICCPLEX{} grow more or less proportional to the numbers of
  applicable lot-types.  The visible outlier for the cpu time of
  \STATICCPLEX{} is instance~107 -- we will see below that it exhibits
  an exceptionally large integrality gap of around~10\,\%.  ASG is not
  equally affected by this problem, most probably because the
  optimality gap of SFA is zero for this instance, as we can see below
  as well.}

\begin{figure}[H]
  \centering
  \includegraphics[width=\linewidth]{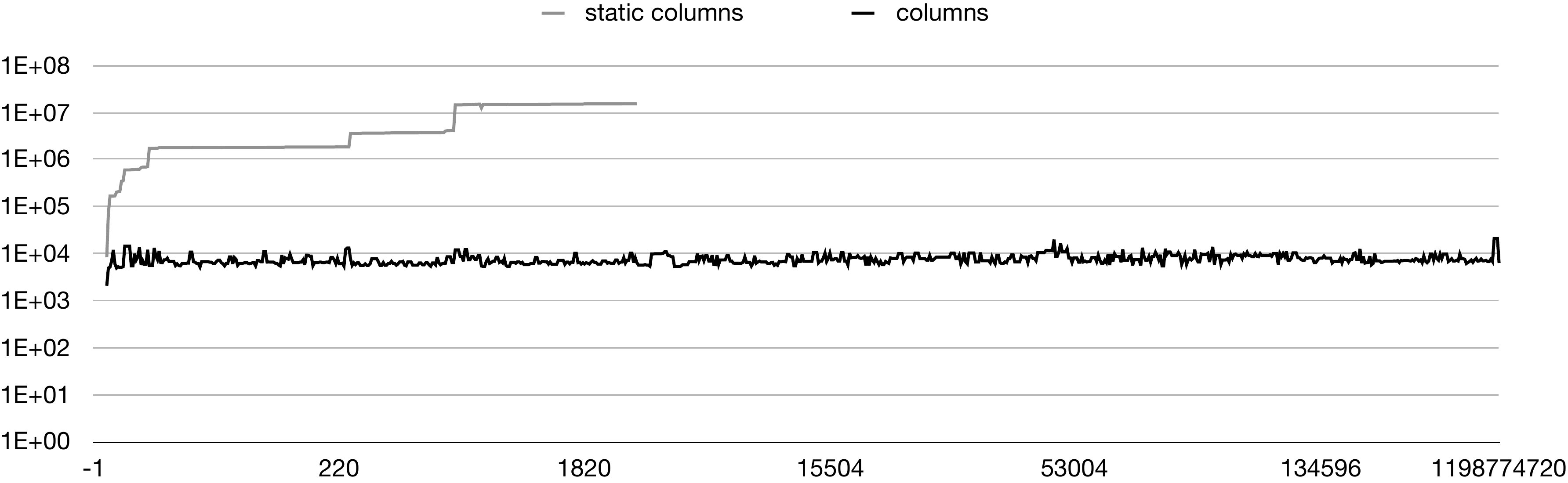}
  \caption{Comparison of columns needed in seconds on 860 real-world
    instances ordered by the numbers of applicable lot-types
    (indicated \replaced[id=JR]{by the grey line}{on the category
      axis})\added[id=JR]{; note that we connected all data points by
      straight lines for better visibility, although we have a
      discrete number of $860$ instances on the category axis}}
  \label{fig:results-real-world-cols}
\end{figure}
The reason for this success can be explained from
Figure~\ref{fig:results-real-world-cols}, where the total numbers of
columns used in the solution procedure is shown on the logarithmic
scale.  It can be seen that ASG needs consistently only around
$10,000$ columns to find and prove the optimum, whereas
\replaced[id=JR, comment={used \STATICCPLEX}]{\STATICCPLEX}{the static
  solution}, of course, uses all columns of the model, which are two
to three orders of magnitude more for the considered instances.
Although we know the numbers of columns that \replaced[id=JR,
comment={used \STATICCPLEX}]{\STATICCPLEX}{the static solution} would
have used if there had been enough memory available, we did not
include these numbers in the chart because the computation did not
take place.

\begin{figure}[H]
  \centering
  \includegraphics[width=\linewidth]{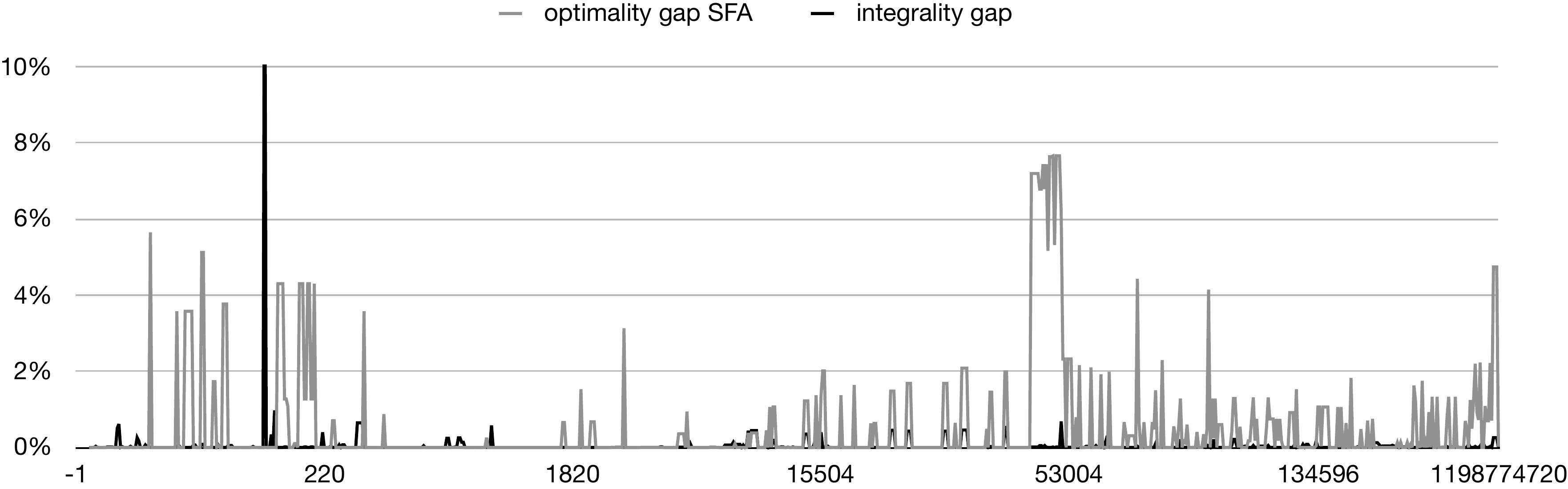}
  \caption{Gaps for the 860 real-world instances ordered by the
    numbers of applicable lot-types\added[id=JR]{; note that we
      connected all data points by straight lines for better
      visibility, although we have a discrete number of $860$
      instances on the category axis}\deleted[id=JR]{(indicated on the
      category axis)}}
  \label{fig:results-real-world-gaps}
\end{figure}
\replaced[id=JR, comment={rewrote the interpretation}]{We wanted to
  find success factors of our model and algorithm.  In our opinion, a
  tight model and a reasonable heuristics are two important success
  factors in a branch-and-price
  method. Figure~\ref{fig:results-real-world-gaps} confirms that the
  integrality gap of our SLDP-model is small throughout with a single
  exception (instance no.~107 with 10\,\%), and also the optimality
  gap of the SFA heuristics is in most cases small (often it is zero,
  e.g., for instance~107 with the large integrality gap).  We have
  also considered the polynomially-sized model in
  Appendix~\ref{sec:comp-model-param}.  Because it exhibits a very
  large integrality gap (see Appendix~\ref{sec:comp-model-param}), we
  dismissed it in favor of the exponential model in this paper.
  Figure~\ref{fig:results-real-world-gaps} provides evidence for the
  fact that our formulation with many variables is quite tight.  The
  computation times show that, in both the smaller and the larger
  instances, our new algorithm ASG can take advantage of this
  increased tightness.}{In order to characterize the difficult
  instances, we evaluated the relative integrality gap and the
  relative gap of the SFA heuristic in
  Figure~\ref{fig:results-real-world-gaps}.  The excellent performance
  really depends on the small integrality gap of the model.  However,
  making a tight model is the whole point of using many variables in
  ILP models.  We see once more that the model tightness is much more
  important than the numbers of columns when it comes to ILP solving
  -- the column generation is very effective.}

\begin{figure}[H]
  \centering
  \includegraphics[width=\linewidth]{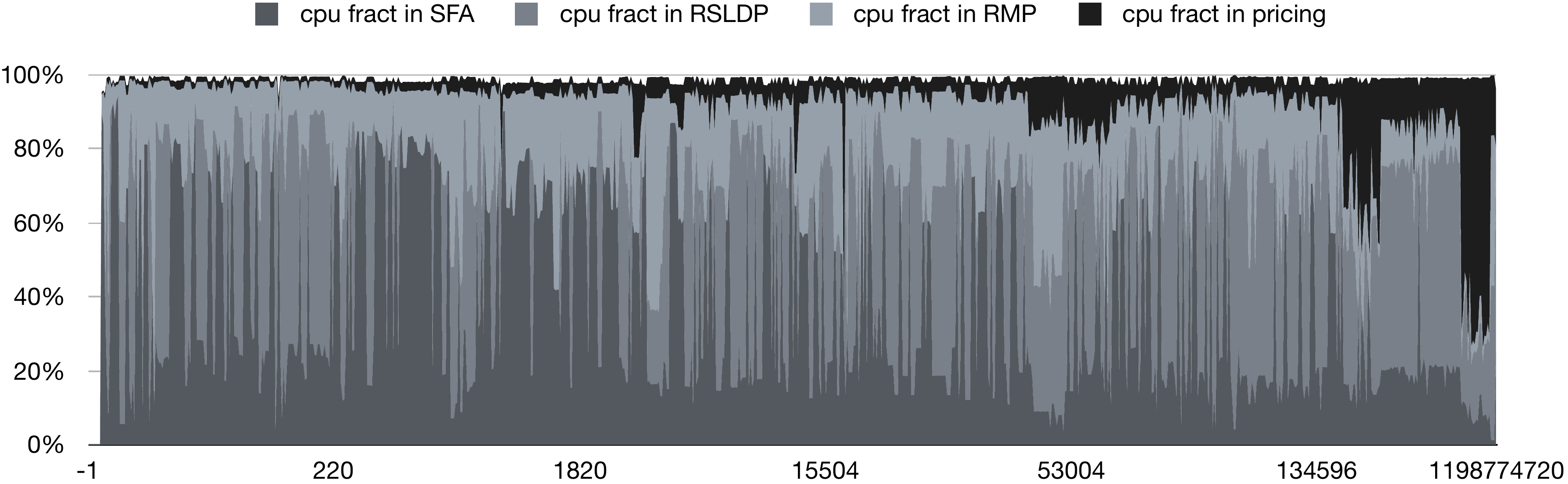}
  \caption{The relative CPU usage for ASG in percent on 860 real-world
    instances ordered by the numbers of applicable
    lot-types\added[id=JR]{; note that we connected all data points by
      straight lines for better visibility, although we have a
      discrete number of $860$ instances on the category
      axis}\deleted[id=JR]{(indicated on the category axis)}}
  \label{fig:results-real-world-cpufract}
\end{figure}
We wanted to know whether ASG has serious bottlenecks.  To this end,
we show in Figure~\ref{fig:results-real-world-cpufract} the fractions
of the CPU time spent in the most important subroutines of ASG.  It
can be seen that while in the fast instances most of the (short) time
is spent in the SFA heuristics, in the slow instances more time is
spent in the RSLDP solving.  Only in the largest instances the
fraction of time spent in the pricing routine becomes dominant.  We
have checked the effectiveness of pricing with respect to the
canonical lifting instead of the characteristic lifting in
Section~\ref{sec:theory}, and substantially more columns have been
generated resulting in larger CPU times.  Thus, the results in
Section~\ref{sec:theory} are beneficial for the success of ASG.

Let us now turn to the results for the $5 \times 9$ random instances.
\added[id=JR, comment={motivated the restriction to the LDP for the
  following}]{Because we already concluded that the difference in
  solving the nominal LDP and the SLDP is minor compared to the
  difference to \STATICCPLEX, we restricted ourselves to the solution
  of randomly generated deterministic LDPs.}

\begin{figure}[H]
  \centering
  \includegraphics[width=\linewidth]{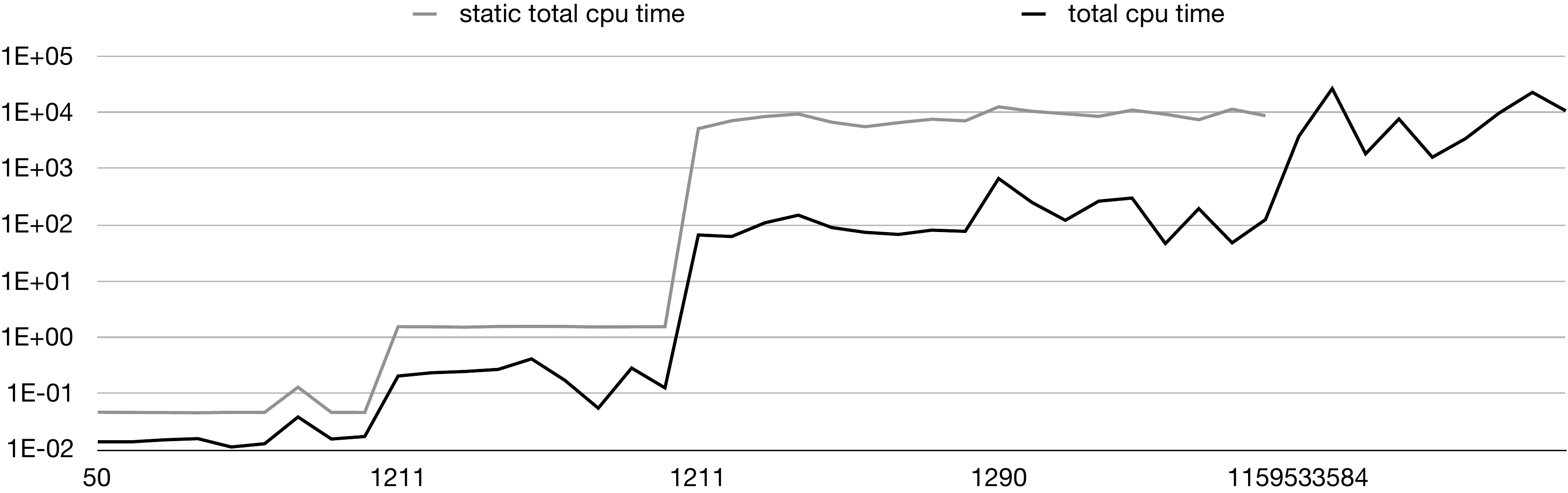}
  \caption{Comparison of CPU times in seconds on $5 \times 9 = 45$
    random instances ordered by the numbers of applicable lot-types
    (indicated \replaced[id=JR]{by the grey line}{on the category
      axis})\added[id=JR]{; note that we connected all data points by
      straight lines for better visibility, although we have a
      discrete number of $5 \times 9$ instances on the category axis}}
  \label{fig:results-random-cpu}
\end{figure}
Figure~\ref{fig:results-random-cpu} shows the CPU times in seconds in
the same way as for the real-world instances.  We see that ASG is one
to two orders of magnitude faster than \replaced[id=JR, comment={used
  \STATICCPLEX}]{\STATICCPLEX}{static solving} and can solve all
\added[id=JR, comment={used the defined notion}]{consistent}
instances.  The instances from the largest group, however, take much
longer this time.  The largest observed CPU time of $26{,}362.7$
seconds was spent in instance~$2$ of the $5$th group.  The performance
difference compared to the real-world instances (some of which have
the about the same size) can be explained.  In the real-world, the
expected demand for sizes is not uniform over all branches.  There is
a hidden imbalance, and $5$ lot-types are enough to cover all the
variations quite well.  In the random instances,
\replaced[id=JR,comment={Is this really correct in your instance
  generator?}]{the expected demands are identical for all sizes in all
  branches.  Thus, any lot-type incurs almost the same cost as all its
  size permutations.  Since for the large random instances with $12$
  sizes the number of permutations leading to a different lot-type is
  much larger than~$k = 5$ (except for lot-types using the same number
  of pieces in every size), there are many distinct lot-type designs
  with almost the same costs.  This makes the long computation times
  plausible.}{all the imbalance is just noise.}  \deleted[id=JR,
comment={removed a remark that noone understood}]{If real-world
  instances could be expected to be like our random instances, then
  probably the lot-type based process would not be used in practice in
  the first place.}  In a sense, the random instances constitute a
stress test for ASG, and we find that the performance compared to
\replaced[id=JR, comment={used \STATICCPLEX}]{\STATICCPLEX}{static
  solving} is still impressive.

\begin{figure}[H]
  \centering
  \includegraphics[width=\linewidth]{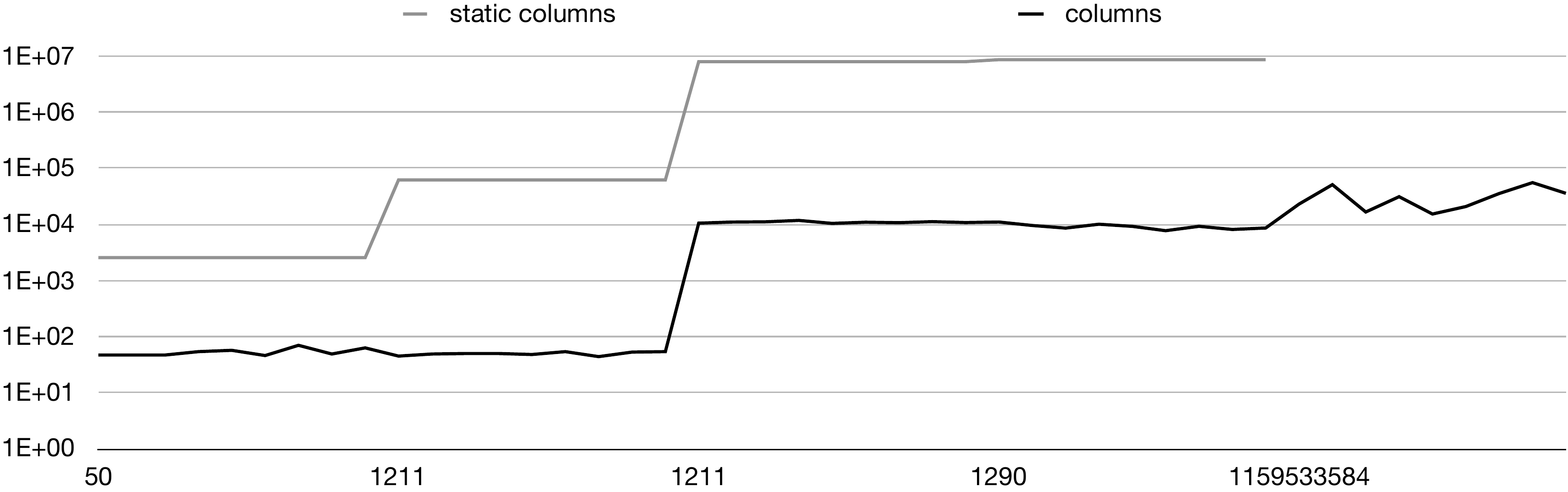}
  \caption{Comparison of columns needed in seconds on
    $5 \times 9 = 45$ random instances ordered by the numbers of
    applicable lot-types (indicated \replaced[id=JR]{by the grey
      line}{on the category axis})\added[id=JR]{; note that we
      connected all data points by straight lines for better
      visibility, although we have a discrete number of $5 \times 9$
      instances on the category axis}}
  \label{fig:results-random-cols}
\end{figure}
The reason for the performance difference between ASG and
\replaced[id=JR, comment={used \STATICCPLEX}]{\STATICCPLEX}{static
  solving} again correlates with the total number of columns used
during the solution process.  Figure~\ref{fig:results-random-cols}
shows that \replaced[id=JR, comment={more standard english}]{the
  number of columns used by ASG is two to three orders of magnitude
  smaller}{ASG gets away (even in the random instances) with two to
  three orders of magnitude fewer columns}.  The numbers of columns
used is stable for the first four groups of instances and starts to
vary in the group of the largest instances.

\begin{figure}[H]
  \centering
  \includegraphics[width=\linewidth]{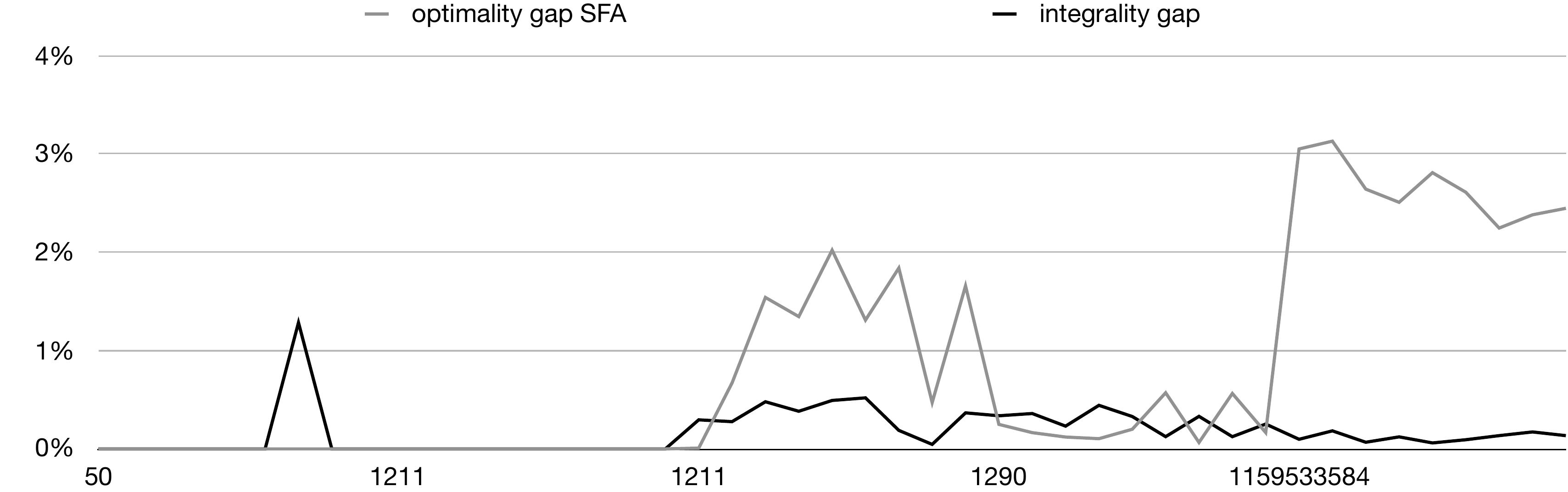}
  \caption{Gaps for the $5 \times 9 = 45$ random instances ordered by
    the numbers of applicable lot-types\added[id=JR]{; note that we
      connected all data points by straight lines for better
      visibility, although we have a discrete number of $5 \times 9$
      instances on the category axis}\deleted[id=JR]{ (indicated on
      the category axis)}}
  \label{fig:results-random-gaps}
\end{figure}
We see in Figure~\ref{fig:results-random-gaps} showing the integrality
gap and the SFA gap that the SFA gap in groups $3$ and $5$ is
substantially larger than in the other groups whereas the integrality
gap is very small \added[id=SK]{throughout} with only very few
exceptions\replaced[id=JR, comment={minor rephrasing}]{. Note,
  however, that the }{, though the} integrality gap is small
everywhere by any practically relevant measure.

\begin{figure}[H]
  \centering
  \includegraphics[width=\linewidth]{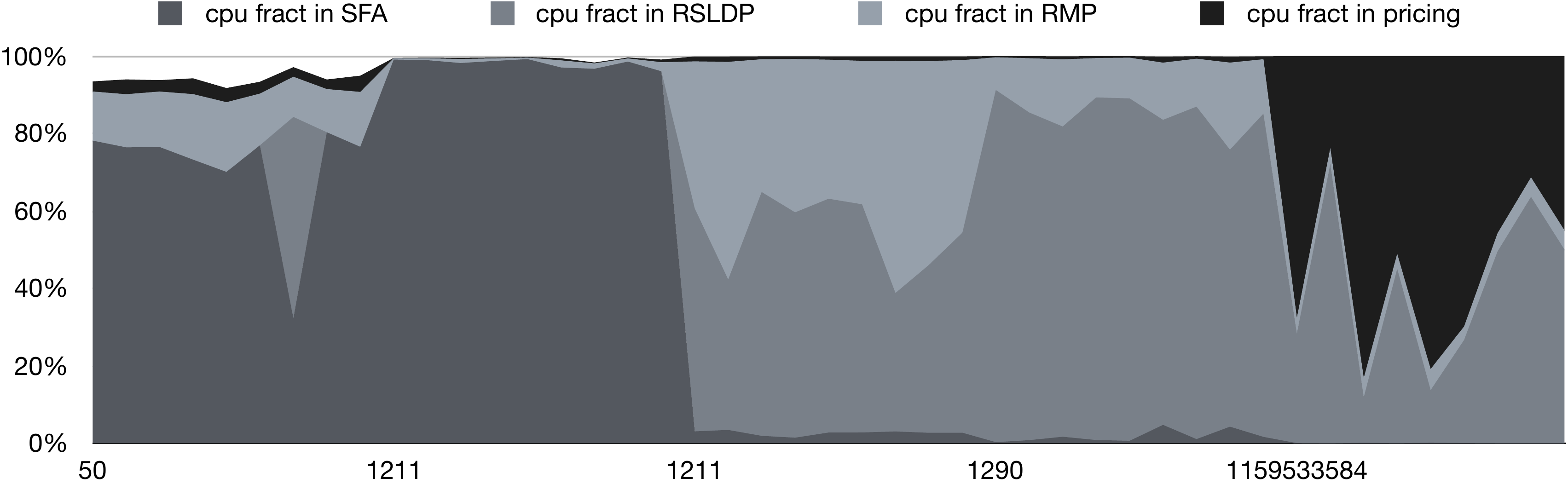}
  \caption{The relative CPU usage for ASG in percent on
    $5 \times 9 = 45$ random instances ordered by the numbers of
    applicable lot-types\added[id=JR]{; note that we connected all
      data points by straight lines for better visibility, although we
      have a discrete number of $5 \times 9$ instances on the category
      axis}\deleted[id=JR]{ (indicated on the category axis)}}
  \label{fig:results-random-cpufract}
\end{figure}
The analysis of possible bottlenecks in
Figure~\ref{fig:results-random-cpufract} shows a similar result as for
the real-world instances.  Only in the very large instances, a
substantial amount of time is spent in pricing.  For the small
instances, the CPU time is dominated by the time limit for SFA.

To sum up, we see that ASG is substantially faster and more
memory-efficient than \replaced[id=JR, comment={used
  \STATICCPLEX}]{\STATICCPLEX}{static solving}, even in smaller
instances.  That is, for this particular problem column generation
should not only be used for the instances that do not fit into the
computer but for \emph{all} instances.

\section{Conclusion and future work}
\label{sec_conclusion}

\noindent
We have considered the stochastic lot-type design problem SLDP, which
is an industrially relevant problem.  We provided a tight ILP model of
it.  This model, however, has in many cases too many variables for
solvers off the shelf.  Thus, we presented a %%custom-made
branch-and-price algorithm ASG that was able to solve a large number
of real-world and random LDP of various scales exactly to proven
optimality orders of magnitude more quickly %%faster
than static ILP-solving.  The pricing routine in ASG benefits
from %%takes profit of
a characteristic lifting of dual variables that provably avoids the
generation of only seemingly promising but actually ineffective
columns.

It would be valuable to find out whether the principles of ASG can be
generalized to a problem independent ILP setting.  The main reason not to
implement branch-and-price is the
amount of effort to close a small gap  
%% large effort for the small gap to close 
after an ILP has been solved based on all the columns necessary for an LP
optimum.  ASG might serve as an intermediate technique whose implementation
effort is reasonable.  Another added bonus of ASG is that all progress in commercial
ILP-solver technology is exploited.

\section*{Acknowledgments}

\added[id=JR, comment={added an acknowledgment}]{We thank two
  anonymous referees for their suggestions that greatly improved the readability of the paper.} 
 %% improved the presentation of the paper by a large amount.}

%% \nocite{cg1}
%% \bibliography{columngeneration_LDP}
%% %% \bibliographystyle{plain}
%% \bibliographystyle{abbrv}

\appendix

\section{A compact model for parameterizable sets of lot-types}
\label{sec:comp-model-param}

\added[id=SK]{For completeness, we want to mention that the large
  number of variables of our ILP formulation from
  Section~\ref{sec:modelling} is far from being inavoidable. In the
  case of a \replaced[id=JR,
  comment={unified}]{parametrized}{parameterizable} set of lot-types,
  as described in Section~\ref{sec_formal_problem_statement}, we can
  model the SLDP with fewer variables. However, the LP relaxation of
  the subsequent compact model yields large integrality gaps.}
\begin{align*}
  \min && \sum_{b\in\mathcal{B}}\sum_{s\in\mathcal{S}}\sum_{a\in\mathcal{A}} \delta_{b,s}^a\\
  s.t.  &&
           \sum_{i\in\mathcal{K}}\sum_{m\in\mathcal{M}} x_{b,i,m} &= 1 && \forall b\in\mathcal{B}\\
       &&
          v_{b,s,i,m} -max_c\cdot x_{b,i,m}& \le 0&& \forall (b,s,m,i)\in\mathcal{U}\\
       &&
          v_{b,s,i,m}-l_{i,s} & \le 0&& \forall (b,s,m,i)\in\mathcal{U}\\
       &&
          v_{b,s,i,m} -l_{i,s}-max_c\cdot x_{b,i,m}& \ge -max_c&& \forall (b,s,m,i)\in\mathcal{U}\\
       &&\!\!\!\!\!\!\!\!\!\!\!\!\!\!\!\!\!\!\!\!\!  \underline{I} \le
          \sum_{b\in\mathcal{B}}\sum_{s\in\mathcal{S}}\sum_{m\in\mathcal{M}}
          \sum_{i\in\mathcal{K}}
          m\cdot v_{b,s,i,m} &\le \overline{I}\\
       &&
          l_{i,s} &\le max_c && \forall i\in\mathcal{K}, s\in\mathcal{S}\\
       &&
          l_{i,s} &\ge min_c && \forall i\in\mathcal{K}, s\in\mathcal{S}\\
       &&
          \sum_{s\in\mathcal{S}} l_{i,s} &\le max_t && i\in\mathcal{K}\\
       &&
          \sum_{s\in\mathcal{S}} l_{i,s} &\ge min_t && \forall i\in\mathcal{K}\\
          % &&
          % \delta_{b,s}^a+m\cdot l_{i,s}-d_{b,s}^a\cdot x_{b,i,m} &\ge 0&&\forall (b,s,m,i)\in\mathcal{U}, a\in\mathcal{A}\\
          % &&
          % \delta_{b,s}^a-m\cdot l_{i,s}+m\cdot max_c\cdot\left(1- x_{b,i,m}\right) &\ge -d_{b,s}^a&&
          % \forall (b,s,m,i)\in\mathcal{U}, a\in\mathcal{A}\\
       &&
          \delta_{b,s}^a+\sum\limits_{i\in\mathcal{K}}\sum\limits_{m\in\mathcal{M}}m\cdot v_{b,s,i,m} &\ge d_{b,s}^a&&
                                                                                                                       \forall b\in\mathcal{B},s\in\mathcal{S},a\in\mathcal{A}\\
       &&
          \delta_{b,s}^a-\sum\limits_{i\in\mathcal{K}}\sum\limits_{m\in\mathcal{M}}m\cdot v_{b,s,i,m} &\ge -d_{b,s}^a&&
                                                                                                                        \forall b\in\mathcal{B},s\in\mathcal{S},a\in\mathcal{A}\\
       &&
          x_{b,i,m} & \in\{0,1\} && \forall b\in\mathcal{B}, i\in\mathcal{K}, m\in\mathcal{M}\\
       &&
          v_{b,s,i,m} & \ge 0 &&\forall (b,s,m,i)\in\mathcal{U}\\
       &&
          l_{i,s} &\in\mathbb{Z}&&\forall i\in\mathcal{K}, s\in\mathcal{S}\\
       && \delta_{b,s}^a &\ge 0&&\forall b\in\mathcal{B},s\in\mathcal{S},
                                  a\in\mathcal{A},
\end{align*}
where we use the abbreviations $\mathcal{K}:=\{1,\dots,k\}$ and
$\mathcal{U}:=\mathcal{B}\times\mathcal{S}\times\mathcal{M}\times\mathcal{K}$.
\added[id=JR, comment={stressed the important difference to the model
  in the main part}]{Obviously, both the number of variables and the
  number of constraints is polynomial in the input parameters, in
  contrast to the model in Section~\ref{sec:modelling}.}
\added[id=JR, comment={introductory sentence}]{The meaning of the
  variables and constraints is as follows.} We utilize the binary
variable $x_{b,i,m}$ to model the assignment of the lot-type and
multiplicity to a certain branch~$b$, i.e., we have $x_{b,i,m}=1$ if
and only if branch~$b$ is supplied with \replaced[id=JR,
comment={minor correction}]{the $i$th selected lot-type}{lot-type~$i$}
in multiplicity~$m$. Of course, for each branch $b$ only one
$x_{b,i,m}$ is one. In order to incorporate the bounds on the total
number of supplied items we utilize the auxiliary variables
$v_{b,s,i,m}$. For a given branch~$b$ and size~$s$ we set
$v_{b,s,i,m}=l_{i,s}\cdot x_{b,i,m}$, i.e.{} $v_{b,s,i,m}=l_{i,s}$ if
branch~$b$ is supplied with lot-type~$i$ in multiplicity~$m$ and
$v_{b,s,i,m}=0$ otherwise. The linearization of this non-linear
equation is quite standard using suitable big-M constants.  The
deviations $\delta_{b,s}^a$ then are given by
% $\delta_{b,s}^a=\left|d_{b,s}^a -m\cdot l_{i,s}\cdot
%   x_{b,i,m}\right|$.
$$\delta_{b,s}^a=\left|d_{b,s}^a-\sum\limits_{i=1}^k\sum\limits_{m\in\mathcal{M}}m\cdot v_{b,s,i,m}\right|.$$
Again the linearization of this non-linear equation is
standard.\footnote{We remark that one can express the deviations
  $\delta_{b,s}^a$ also using the binary assignment variables
  $x_{b,i,m}$ instead of the item counts $v_{b,s,i,m}$. In this case
  we have to replace the two inequalities containing $\delta_{b,s}^a$
  by $\delta_{b,s}^a+m\cdot l_{i,s}-d_{b,s}^a\cdot x_{b,i,m} \ge 0$
  and
  $\delta_{b,s}^a-m\cdot l_{i,s}+m\cdot max_c\cdot\left(1-
    x_{b,i,m}\right) \ge -d_{b,s}^a$ for all
  $(b,s,m,i)\in\mathcal{U}, a\in\mathcal{A}$.  Or we may use both sets
  of inequalities. \deleted[id=JR]{It turns out that this alternative
    formulations is solved slightly faster on some inspected
    instances.}}

Since the symmetric group on $k$ elements acts on the $k$ different
lot-types one should additionally destroy the symmetry in the stated
ILP formulation. This can be achieved by assuming that the lot-types
$\left(l_{i,s}\right)_{s\in\mathcal{S}}$ are lexicographically
ordered. For the ease of notation we assume that the set of
sizes~$\mathcal{S}$ is given by $\{1,\dots,s\}$. As an abbreviation we
set $t:=\max_c-\min_c+1$. With this the additional inequalities
\begin{align*}
  && \sum_{j=1}^s \left(l_{i,j}-\min_c\right)\cdot t^{s-j} &\ge
                                                             1+\sum_{j=1}^s \left(l_{i+1,j}-\min_c\right)\cdot t^{s-j} &&\forall
                                                                                                                          1\le i\le k-1,
\end{align*}
which are equivalent to
$\sum_{j=1}^s \left(l_{i,j}-l_{i+1,j}\right)\cdot t^{s-j}\ge 1$
for all $1\le i\le k-1$, will do the job. We remark that using some
additional auxiliary variables a numerical stable variant can be
stated easily. (This becomes indispensable if $t^s$ gets too large.)

\medskip

We remark that the LP relaxation of the compact model yields large
integrality gaps~\footnote{Assuming some technical conditions on the
  demands, the four parameters for the lot-types, and the applicable
  multiplicities, which are not too restricting, one can construct a
  solution of the LP relaxation having an objective value of zero.
  E.g.\ we assume $\min_c\le d_{b,s}^a\le \max_c\cdot\max\{m\in\mathcal{M}\}$}.
The typical approach would now be to solve a
Dantzig-Wolfe decomposition of the compact model by dynamic column
generation, leading to a master problem very similar to the SLDP model
we presented in the first place.  And column generation on that SLDP
model is what we proposed in the paper.

% Implementierung der Formel in Maple:
% ==================================== Die direkte Variante: f := proc
% (a, b, c, d) options operator, arrow; sum(sum((-1)^i*binomial(d,
% i)*binomial(n-i*(a+1)+d-1, d-1), i = 0 ..  min(floor(n/(a+1)), d)),
% n = b .. c) end proc; fuehrt leider zu einem Fehler: > f(4, 12, 30,
% 12); Error, (in SumTools:-DefiniteSum:-ClosedForm) summand is
% singular in the interval of summation
% 
% Deswegen mit einem weiteren Zwischenschritt implementiert.
% Hilfsfunktion: g := proc (a, b, d) options operator, arrow;
% sum((-1)^i* binomial(d,i)*binomial(b-i*(a+1)+d-1, d-1), i = 0 ..
% min(floor(b/(a+1)), d)) end proc;
% 
% Eingabe der fuenf Parameter p1,...,p5, Umrechnung in 4 Parameter und
% Aufsummieren der Faelle aus der Hilfsfunktion g: p1 := 0; p2 := 5;
% p3 := 12; p4 := 30; p5 := 12; a := p2-p1; b := p3-p1*p5; c :=
% p4-p1*p5; d := p5; s := 0; for j from b to c do s := s+g(a, j, d)
% end do; s;

% Zur empirischen Ueberpruefung der Formel wurde auch eine explizite
% Generierung der Lottypen in C++ implementiert. Z.B. in
% count_lots.cpp bzw. in Miriams Paketen.

\end{document}